\documentclass[12pt]{amsart}

\usepackage{amsfonts,amsmath,amssymb,amsthm}
\usepackage{hyperref}
\hypersetup{breaklinks=true}
\usepackage{rotating}
\usepackage{array}

\numberwithin{equation}{section}
\theoremstyle{plain}
\newtheorem{theorem}[equation]{Theorem}  
\newtheorem{lemma}[equation]{Lemma}       
\newtheorem{proposition}[equation]{Proposition}      
\newtheorem{corollary}[equation]{Corollary}

\newtheorem{question}[equation]{Question}

\newcolumntype{M}[1]{>{\raggedright}m{#1}}

\theoremstyle{remark}
\newtheorem*{remark*}{Remarks}  
\newtheorem*{example*}{Example}

\theoremstyle{definition}
\newtheorem{definition}[equation]{Definition}
\newtheorem{remark}[equation]{Remark}
\newtheorem{example}[equation]{Example}

\newcommand{\im}{\operatorname{Im}}
\newcommand{\Isom}{\mathop{\mathrm{Isom}}\nolimits}
\newcommand{\Ker}{\operatorname{Ker}}

\begin{document}

\title[QI-invariance of group $L^p$-cohomology and vanishings]
{Quasi-isometric invariance\\ of continuous group $L^p$-cohomology,\\ and first applications to vanishings}
 
\author{Marc Bourdon and Bertrand R\'emy} 
\maketitle

\begin{abstract}
We show that the continuous $L^p$-cohomology of locally compact second countable groups is a quasi-isometric invariant.
As an application, we prove partial results supporting a positive answer to a question asked by M.~Gromov, suggesting a classical behaviour of continuous $L^p$-cohomology of simple real Lie groups. 
In addition to quasi-isometric invariance, the ingredients are a spectral sequence argument and Pansu's vanishing results for real hyperbolic spaces. 
In the best adapted cases of simple Lie groups, we obtain nearly half of the relevant vanishings. 

\vspace*{2mm} \noindent{2010 Mathematics Subject Classification: } 20J05, 20J06, 22E15, 22E41, 53C35, 55B35, 57T10, 57T15.
%
%
%
%

\vspace*{2mm} \noindent{Keywords and Phrases: } $L^p$-cohomology, topological group, Lie group, symmetric space, quasi-isometric invariance, spectral sequence, cohomology vanishing, root system.
\end{abstract}

\setlength{\parskip}{\smallskipamount}

\tableofcontents

\setlength{\parskip}{\medskipamount}
\newpage

\section{Introduction}
\label{1}

Let $G$ be a locally compact second countable group. 
Then $G$ is metrizable and countable at infinity, i.e. is a countable union of compact subsets \cite[IX \textsection 2.9, Cor.\! to Prop.\! 16, p.\! 21]{BBK-TG-5to10}. 
Therefore we can use Fubini's theorem on $G$ and the group admits a left-invariant proper metric defining its topology \cite[Struble Th.\! 2.B.4]{CdH}.
In addition, it is also equipped with a left invariant Haar measure \cite[VII \textsection 1.2, Th.\! 1]{BBK-INT-7}, which enables us to define the spaces $L^p(G)$ 
\cite[IV \textsection 3.4, D\'ef.\! 2]{BBK-INT-1to4}.

In this paper we study the \emph{group $L^p$-cohomology} of $G$ for $p>1$.
By definition, this is the continuous cohomology of $G$ \cite[Chap.\! IX]{BW}, with coefficients in the right-regular representation on $L^p(G)$.
It shall be denoted by ${\rm H}^*_{\mathrm {ct}} \bigl( G, L^p(G) \bigr)$.
We also consider the associated \emph{reduced cohomology}, denoted by $\overline{{\rm H}^*_{\mathrm {ct}}} \bigl( G, L^p(G) \bigr)$ 
(they are the largest Hausdorff quotients of the previous spaces -- see Definition \ref{ell_pcont_def}).

Our first result is the quasi-isometry invariance of these cohomology spaces (see Corollary \ref{cont_asymp_cor} for details).

\begin{theorem}
\label{intro-thm1}
Let $G_1$, $G_2$ be locally compact second countable groups, equipped with left-invariant proper metrics.
Every quasi-isometry $F: G_1 \to G_2$ induces canonically an isomorphism of graded topological vector spaces $F ^* :{\rm H}^*_{\mathrm {ct}} \bigl( G_2, L^p(G_2) \bigr) \to {\rm H}^*_{\mathrm {ct}} \bigl( G_1, L^p(G_1) \bigr)$.
The same holds for the reduced cohomology.
\end{theorem}

For example the group $L^p$-cohomology of a group $G$ as above is isomorphic to the group $L^p$-cohomology of any of its cocompact lattices.
When $G$ is semisimple, its group $L^p$-cohomology is isomorphic to the group $L^p$-cohomology of any of its parabolic subgroups; this generalizes in fact to (the points of) any connected real algebraic group. 

Theorem \ref{intro-thm1} was already known for the finite $K(\pi, 1)$-groups by M.\! Gromov (see \cite[p.\! 219]{G} and Remark \ref{simplicial_remark}), and then proved for finitely generated groups by G.\! Elek \cite{Elek}. 
In degree $1$, the theorem is due to Cornulier-Tessera \cite{CT}. 
We prove it (in any degree) by relating the group $L^p$-cohomology with the asymptotic $L^p$-cohomology.
The  latter has been defined by Pansu, and it is known to be a quasi-isometry invariant \cite{Pa95}. The equivalence between
the group $L^p$-cohomology and the asymptotic $L^p$-cohomology was established by G.\! Elek \cite{Elek} for finitely generated groups, and by R.\! Tessera \cite{T} in degree $1$.
We recently learnt from R.~Sauer and M.~Schr\"odl that they established the coarse invariance of vanishing of $\ell^2$-Betti numbers for unimodular locally compact second countable groups \cite{SaSc}. To prove this result they established a coarse equivalence version of Theorem \ref{intro-thm1}, by using the same comparison stategy  -- see Theorem 12 in \cite{SaSc}. 

Our motivation for Theorem \ref{intro-thm1}, which leads to the second main theorem of this paper, is understanding the $L^p$-cohomology of semisimple groups of real rank $\geqslant 2$ 
(for the precise conventions on Lie groups in the paper, we refer to the end of this introduction). 
In fact, we have in mind the following question, asked by M.~Gromov \cite[p.\! 253]{G}.

\begin{question}
Let $G$ be a simple real Lie group.
We assume that $l=\mathrm{rk}_{\bf R} (G) \geqslant 2$. 
Let $k$ be an integer $<l$ and $p$ be a real number $>1$. 
Do we have: ${\rm H}^k_{\mathrm {ct}} \bigl(G, L^p(G) \bigr) = \{0\}$?
\end{question}

The previous theorem allows one to reduce the study of the group $L^p$-cohomology of a Lie group $G$ as in the question, to the group $L^p$-cohomology of any of its parabolic subgroups. 
Such a parabolic subgroup, say $P$, admits a Levi decomposition, which is a semi-direct product decomposition $P = M \ltimes AN$ where $M$ is semi-simple and $AN$ is the (solvable) radical of $P$. 
The class for which we can fruitfully use Theorem \ref{intro-thm1} consists of the simple Lie groups in which some maximal parabolic subgroup has suitable geometric properties. By lack of better terminology, we will call them \emph{admissible}. 
Here is their definition:

\begin{definition}\label{intro-def}
A simple real Lie group is called \emph{admissible} if it admits a parabolic subgroup whose radical is quasi-isometric to a real hyperbolic space.
\end{definition}

\begin{theorem}
\label{intro-thm2}
Let $G$ be an admissible simple real Lie group, and let $d$ be the dimension of the corresponding real hyperbolic space.
Let also $D$ denote the dimension of the Riemannian symmetric space attached to $G$. 
\begin{enumerate}
\item We have: ${\rm H}^0 _{\mathrm{ct}}\bigl( G, L^p(G) \bigr) = \{0\}$, and also: ${\rm H}^k _{\mathrm{ct}}\bigl( G, L^p(G) \bigr) = \{0\}$ for $k \geqslant D$.
\item Suppose now that $k \in \{1, ... ,D-1\}$. 
Then we have:
$${\rm H}^k _{\mathrm{ct}}\bigl( G, L^p(G) \bigr) = \{0\} \,\, \hbox{\rm for} \,\, k \leqslant \frac{d-1}{p} \,\, \hbox{\rm and for} \,\, k \geqslant \frac{d-1}{p} + D-d+2.$$
\end{enumerate}

\end{theorem}

Our results also deal with reduced cohomology -- see Theorem \ref{theorem_vanishing} for details.

The prototype of a simple Lie group, namely ${\rm SL}_n({\bf R})$, is admissible, and in this case we obtain (see Example {\ref{ex_SL} for details): 

\begin{corollary}
\label{intro_cor1}
For $G = {\rm SL}_n({\bf R})$, one has ${\rm H}^k _{\mathrm {ct}} \bigl( G, L^p(G) \bigr) = \{0\}$ for $k \leqslant \lfloor \frac{n^2}{4}\rfloor \cdot\frac{1}{p}$
and for $k \geqslant \lfloor \frac{n^2}{4}\rfloor \cdot\frac{1}{p} + \lfloor \frac{(n+1)^2}{4}\rfloor $.
\end{corollary}

We also note that the same circle of ideas can be used to prove the following result, which is a particular case of a general result of Pansu \cite{P2} and 
Cornulier-Tessera \cite{CT}, saying that ${\rm H}^1_{\mathrm {ct}} \bigl( G, L^p(G) \bigr) = \{0\}$ for every $p>1$ and
every connected Lie group $G$, unless $G$ is Gromov hyperbolic or amenable unimodular. 

\begin{corollary}
\label{intro_cor2}
For every admissible simple Lie group $G$ such that $\mathrm{rk}_{\bf R} (G) \geqslant 2$ and every $p >1$, we have: ${\rm H}^1 _{\mathrm {ct}} \bigl( G, L^p(G) \bigr) = \{0\}$.
\end{corollary}

This result also follows from the fixed point property for continuous affine isometric actions of semi-simple groups of real rank $\geqslant 2$ on $L^p$-spaces, which is established in \cite{BFGM}.

Let us finish this introduction by describing roughly our method and by discussing its efficiency and its validity. 

The main idea is to combine tools from homological algebra, namely the Hochschild-Serre spectral sequence \cite[IX.4]{BW}, together with geometric results such as Pansu's description of the $L^p$-cohomology of real hyperbolic spaces \cite{Pa99, P1}.
We also use resolutions already considered in \cite{Bl}, whose measurable nature makes them flexible enough to describe nicely intermediate spaces involved in the spectral sequences. 

The general outcome is a new description of the $L^p$-cohomology of $G$ -- see Theorem \ref{theorem_vanishing} and Remark \ref{remark_SL}. 

In order to discuss the efficiency of the method, let us reformulate Theorem \ref{intro-thm2} as follows: for each admissible group $G$ and each exponent $p>1$, we exhibit an interval of degrees, whose explicit length does not depend on $p$ and outside of which vanishing of $L^p$-cohomology is guaranteed (\emph{i.e.}, this interval is a strip of potential non-vanishing of $L^p$-cohomology). 
First, by Lie-theoretic considerations, we can provide the full list of admissible simple Lie groups. 
It contains several infinite sequences of groups, say $\{ G_l \}_{l \geqslant 2}$ where $l=\mathrm{rk}_{\bf R} (G)$. 
In the best cases, the dimension $D_l$ of the symmetric space of $G_l$ and the dimension $d_l$ of the radical of the suitable parabolic subgroup are quadratic polynomials in $l$,
and we have $\displaystyle \lim_{l \to \infty} {d_l \over D_l} = {1 \over 2}$.
Since the width of the strip of potential non-vanishing is $D_l-d_l+2$, this makes us say that we obtain in the best cases of admissible groups nearly half of the relevant vanishings. 

At last, our results can also be compared 
with Borel's result \cite{Borel} about the $L^2$-cohomology of connected semi-simple Lie groups $G$, which makes heavy use of representation theory while we use a more geometric approach.
It  asserts that: 
\begin{itemize}
\item $\overline{{\rm H}^k_{\mathrm {ct}}} \bigl( G, L^2(G) \bigr) = 0$ unless $k = \frac{D}{2}$,
\item  ${\rm H}^k_{\mathrm {ct}} \bigl( G, L^2(G) \bigr) \neq 0$ at least for $k \in \bigl(\frac{D}{2} - \frac{l_0}{2}, \frac{D}{2} +\frac{l_0}{2}\bigr]$,
\end{itemize}
where $D$ is the dimension of the Riemannian symmetric space $G/K$, and
where $l_0$ is the difference between the complex rank of $G$ and the complex rank of $K$. 
In the case $G = {\rm SL}_n({\bf R})$, one has $D = \frac{n^2 + n -2}{2}$ and 
$l_0 = \lfloor \frac{n-1}{2} \rfloor$. 
See also \cite{BFS} for related results about vanishing of the reduced $L^2$-cohomology.

\subsection*{Conventions on Lie groups}

Let us collect here our conventions on Lie groups. 
In what follows, all Lie groups are assumed to be connected. 
In fact, we assume that any semisimple group $G$ we consider here is obtained as the identity component of the real points of an algebraic semisimple ${\bf R}$-group ${\bf G}$, assumed to connected as a linear algebraic group. 
In other words, we have: $G = {\bf G}({\bf R})^\circ$. 
Recall that if ${\bf H}$ is any connected linear algebraic group defined over ${\bf R}$, then the group ${\bf H}({\bf R})^\circ$ is a finite index normal subgroup in ${\bf H}({\bf R})$  \cite[14.1]{BoTi}. 
Moreover if ${\bf S}$ is a maximal ${\bf R}$-split torus in ${\bf G}$ as before, that is a subgroup isomorphic over ${\bf R}$ to some power ${\mathbb G}_m^d$ of the multiplicative group ${\mathbb G}_m$, then we have: ${\bf G}({\bf R}) = G \cdot {\bf S}({\bf R})$  \cite[Th.\! 14.4]{BoTi}. 
The classical notation for the identity component ${\bf S}({\bf R})^\circ$ is $A$; the latter group is a maximal subgroup diagonalizable over ${\bf R}$ in $G$ and, in view of its construction, we will still call it a maximal ${\bf R}$-split torus. 
It is checked in \cite[14.7]{BoTi} that, in the case of semisimple groups as in our convention, the root system and the Weyl group arising from considerations from algebraic groups, \emph{e.g.}\! as in \cite[\textsection 21]{Borel-LAG}, coincide with the root system and the Weyl group as defined by \'E.~Cartan in terms of symmetric spaces -- see also \cite[\textsection 24.C]{Borel-LAG}.

\subsection*{Remarks and Questions}

\textbf{1)} The definition of an admissible simple Lie group (Definition \ref{intro-def}) 
generalizes to semisimple Lie groups. It is straightforward that a semisimple real Lie group is admissible if and only if one of its simple factors is. Theorem \ref{intro-thm2} is still valid for admissible semisimple Lie groups (the proof is the same). 
However we think
that some $L^p$-version of the K\"unneth formula may be true for products of groups and could give better results. More precisely, suppose that $G_1$, $G_2$ are locally compact second countable groups, such that for some $p \in (1, +\infty)$ and  $n_1, n_2 \in \bf N$, 
one has ${\rm H}^k_{\mathrm {ct}} \bigl( G_i, L^p(G_i) \bigr) =0$ for every $k\le n_i$. Is ${\rm H}^k _{\mathrm {ct}} \bigl( G_1 \times G_2, L^p(G_1 \times G_2) \bigr) =0$  for every $k \le n_1 + n_2 +1$ ? We remark that these problems are discussed in  \cite[p.\! 252]{G}.

\textbf{2)} Even if a simple real Lie group
is non-admissible, the solvable radical of any maximal proper parabolic subgroup 
decomposes as ${\bf R} \ltimes N$,
where $N$ is a nilpotent Lie group on which ${\bf R}$ acts by contracting diffeomorphisms. Such a group belongs to the class of the so-called 
\emph{Heintze groups}, \emph{i.e.\!} the (connected) Lie groups which admit a left-invariant Riemannian metric of negative curvature. 
The $L^p$-cohomology of negatively curved manifolds has been investigated by P.\! Pansu \cite{Pa99, P1} (see also \cite{B16}). 
However, apart from the real hyperbolic spaces, only a partial description is known.
To keep the exposition of the paper as simple as possible, we have choosen to restrict ourself to the class of admissible simple Lie groups.

\subsection*{Organization of the paper} Section \ref{2} is a brief presentation of some standard topics in cohomology of groups.
Section \ref{3} discusses group $L^p$-cohomology: we first consider the discrete group case and recall  some classical connections with the simplicial  $L^p$-cohomology; in the general case, we state the existence of an isomorphism between the group $L^p$-cohomology and the asymptotic $L^p$-cohomology (Theorem \ref{cont_asymp_thm}).
This result, whose proof is given in Section \ref{4}, implies the first quasi-isometric invariance theorem (Theorem \ref{intro-thm1} of the introduction). 
Section \ref{5} contains some preliminary results on the group $L^p$-cohomology of semi-direct products via the Hochschild-Serre spectral sequence.   
Section \ref{6} focuses on applications of the previous sections to the group $L^p$-cohomology of admissible simple real Lie groups  (an improved version of Theorem \ref{intro-thm2} of the introduction). Corollary \ref{intro_cor1} is proved in this section.
Section \ref{7} provides the list of admissible groups and Section \ref{8} gives tables in which the numerical efficiency of the method can be discussed. 
Finally,  in Section \ref{9} we prove Corollary \ref{intro_cor2} on vanishing in degree 1.

\subsection*{Acknowledgements} 
We thank Nicolas Monod, whose paper \cite{Monod} was a valuable source of inspiration.
We also thank Pierre Pansu: the present paper elaborates on  several of his results. 
At last, we thank Y. Cornulier for useful comments and suggestions, and R. Sauer and M. Schr\" odl for drawing our attention to their paper \cite{SaSc}. We are grateful to the anonymous referees for their help in improving the exposition.
M.B.\! was partially supported by the Labex Cempi and B.R. by the GeoLie project (ANR-15-CE40-0012). 
Both authors were supported by the GDSous/GSG project (ANR-12-BS01-0003).

\section{Continuous cohomology}
\label{2}

We review some of the basic notions of the continuous cohomology of topological groups. 
We refer to \cite[Chap.\! IX]{BW} or to \cite{Gui}, and to the references they contain, for more details. 

For topological spaces $X$ and $Y$, we denote by $C(X,Y)$ the set of continuous mappings from $X$ to $Y$, equipped with the compact open topology. 
 
Let $G$ be a locally compact second countable group, and let $(\pi, V)$ be a \emph{topological $G$-module} \emph{i.e.\!} a Hausdorff locally convex vector space over ${\bf R}$ on which $G$ acts via a continuous representation $\pi$. 
In this case, we denote by $V^G \subset V$ the subspace of $\pi(G)$-invariant vectors.
A map $f: A \to B$ between topological $G$-modules
is called a \emph{$G$-morphism} if it is a continuous $G$-equivariant linear map.

For $k \in {\bf N}$, let $C^k (G, V):= C(G^{k+1}, V)$ be the set of continuous maps from $G^{k+1}$ to $V$ equipped with the compact open topology. 
Then $C^k (G,V)$ is a topological $G$-module by means of the following action:
for $g, x_0, ... , x_k \in G$,
$$(g \cdot f) (x_0, ... , x_k) = \pi(g) \bigl(f(g ^{-1} x_0, ... , g^{-1} x_k)\bigr).$$

Consider the following complex 
$$C^0(G,V)^G \stackrel{d_0}{\to} C^1(G,V)^G \stackrel{d_1}{\to} C^2(G,V)^G \stackrel{d_2}{\to} ... $$
where, for $k \in {\bf N}$ and $x_0, ... , x_{k+1} \in G$,
\begin{equation}\label{differential} 
(d _k f)(x_0, ... , x_{k+1}) = \sum _{i =0} ^{k+1} (-1)^{i} f(x_0 , ... , \hat{x_i}, ... , x_{k+1}).
\end{equation}

We first define the cohomologies we are interested in thanks to the above standard homogeneous resolution. 

\begin{definition}\label{ell_pcont_def}
\begin{enumerate}
\item The \emph{continuous cohomology of $G$ with coefficients in $(\pi ,V)$} is the cohomology of this complex. 
It is the collection of groups ${\rm H}_{\mathrm{ct}} ^* (G, V) = \{{\rm H}_{\mathrm{ct}} ^k (G, V) \}_{k \in {\bf N}}$, where
${\rm H}_{\mathrm{ct}} ^0 (G, V) = \Ker d_0 \simeq V^G$ and
${\rm H}_{\mathrm{ct}} ^k (G, V) = \Ker d_k / \im d_{k-1}$ 
for $k \geqslant 1$. 
\item The \emph{reduced continuous cohomology} is the collection $\overline{{\rm H}_{\mathrm{ct}} ^*} (G, V) = 
\{\overline{{\rm H}_{\mathrm{ct}} ^k} (G, V) \}_{k \in {\bf N}}$, where
$\overline{{\rm H}_{\mathrm{ct}} ^0} (G, V) = \Ker d_0 \simeq V^G$, and where
$\overline{{\rm H}_{\mathrm{ct}} ^k} (G, V) = \Ker d_k / \overline{\im d_{k-1}}$ for $k \geqslant 1$ 
(here $\overline{\im d_{k-1}}$ denotes the closure of $\im d_{k-1}$ in $C^k (G,V)$) .
\item When $G$ is a discrete group, we omit the subscript ``$\mathrm{ct}$''  and simply write ${\rm H} ^k (G, V)$ and 
$\overline{{\rm H} ^k} (G, V)$.
\end{enumerate}
\end{definition}

We now recall some notions in relative homological algebra, see \cite[Chap.\! I, \textsection 2 and Chap.\! III, \textsection 1]{Gui} and \cite[Chap.\! IX, 1.5]{BW}. 
\begin{itemize}
\item If $A$ and $B$ are topological $G$-modules as above, we say that a $G$-morphism $f : A \to B$ is {\it strong}~if its kernel and its image are closed direct topological summands in $A$ and $B$ respectively, and if $f$ is strict in the sense that it induces a topological isomorphism between $A/{\rm Ker}(f)$ and ${\rm Im}(f)$;  in \cite[Chap.\! IX, 1.5, p.\! 172]{BW}, such a map $f$ is also called an s-{\it morphism}.  
\item An exact sequence consisting of $G$-morphisms between topological $G$-modules is called {\it strong}~if all the involved maps are strong~; in  \cite[Chap.\! IX, 1.5, p.\! 172]{BW}, such an exact sequence is also called an s-{\it exact sequence}. 
\item A topological $G$-module $U$ is called \emph{relatively injective} if for every strong injection $0 \to A \to B$ between topological $G$-modules, 
every $G$-morphism $A \to U$ extends to a $G$-morphism $B \to U$; in \cite[Chap.\! IX, 1.5, p.\! 172]{BW}, such a module is also called s-{\it injective}. 

\item A complex of topological $G$-modules and $G$-morphisms 
$$0 \to V \stackrel{d_{-1}}{\to} A^0 \stackrel{d_0}{\to} A^1 \stackrel{d_1}{\to} A^2 \stackrel{d_2}{\to} ... $$
is a \emph{strong $G$-resolution of $V$} if for every $k \in {\bf N}$, 
there exists a continuous linear map $h_k: A^k \to A^{k-1}$ (not requested to be a $G$-morphism) such that 
$h _0 \circ d_{-1} = \mathrm{id}$ and for every $k \in {\bf N}$ 
$$d_{k-1} \circ h_k + h_{k+1} \circ d_k = \mathrm{id}.$$
 \item A \emph{relatively injective strong $G$-resolution of $V$} is a strong
$G$-resolution $0 \to V \stackrel{d_{-1}}{\to} A^0 \stackrel{d_0}{\to} A^1 \stackrel{d_1}{\to} A^2 \stackrel{d_2}{\to} ...$
in which the $A^{k}$'s are relatively injective (in Borel-Wallach's terminology, it is an s-exact sequence in which the topological $G$-modules are s-injective).
\end{itemize}

Let $A^* = (A^0 \stackrel{d_0}{\to} A^1 \stackrel{d_1}{\to} A^2 \stackrel{d_2}{\to} ...)$
and $B^* = (B^0 \stackrel{d_0}{\to} B^1 \stackrel{d_1}{\to} B^2 \stackrel{d_2}{\to} ...)$ 
be complexes of Hausdorff locally convex vector spaces.
\begin{itemize}
\item They are  \emph{homotopy equivalent} if there exist continuous homomorphisms 
$\varphi_* : A^*\to B^*$, $\psi_* : B^* \to A^*$, 
$\alpha_* : A^* \to A^{*-1}$, $\beta_* : B^* \to B^{*-1}$, such that $\varphi_*$ and $\psi_*$ commute with $d_*$, and such that
for every $k \in {\bf N}$:
\medskip
\begin{itemize}
\item[ ] $d_{k-1} \circ \alpha_k + \alpha_{k+1}  \circ d_k = \mathrm{id} - \psi_k  \circ \varphi_k$ 
\item[ ] $d _{k-1}\circ \beta_k + \beta_{k+1}  \circ d_k = \mathrm{id} - \varphi_k  \circ \psi_k$,
\end{itemize}
\medskip
with the convention $A^{-1} = B^{-1} = 0$ and $d_{-1} = \alpha _0 = \beta _0 = 0$.
\end{itemize}

\begin{example}\label{bar_ex} The complex 
$$0 \to V \stackrel{d_{-1}}{\to} C^0(G,V) \stackrel{d_0}{\to} C^1(G,V)  \stackrel{d_1}{\to} C^2(G,V) \stackrel{d_2}{\to} ... $$
is a relatively injective strong $G$-resolution of $V$. 
First it is a strong resolution: one defines the $h_k$'s by $(h_k f) (x_0, ... , x_{n-1}) = f(1, x_0, ... , x_{n-1})$.
Secondly, $C^k(G, V)$ is isomorphic to $C^0\bigl(G, C^{k-1} (G,V)\bigr)$; and for every topological $G$-module $(\pi, W)$
the module $C^0(G, W)$ is relatively injective. 
Indeed given a $G$-injection $0 \to A \stackrel{i}{\to} B$ and a continuous linear map
$h: B \to A$ such that $h \circ i = \mathrm{id}$, every $G$-morphism $f: A \to C^0(G, W)$ extends to $\overline{f}: B \to C^0(G, W)$
by letting  for $b \in B$ and $x \in G$:
$$\overline{f} (b)(x)  = \pi (x) \Bigl(f\bigl(h(x^{-1}b)\bigr)(1)\Bigr).$$
\end{example}

\begin{example}\label{simplicial_ex} Suppose $G$ is a discrete group that acts properly discontinuously and freely
on a contractible locally finite simplicial complex $X$ by simplicial automorphisms.
Let $X^{(k)}$ be the set of $k$-simplices of $X$
and let $C^k (X, V)$ be the set of linear maps from the vector space 
${\bf R} X^{(k+1)}$ to $V$. Let $\delta _k: C^k (X, V) \to C^{k+1} (X,V)$
be defined as follows; for $f \in C^k (X, V)$, and $\sigma \in X^{(k+1)}$, 
$$(\delta _k f)(\sigma) = f(\partial \sigma).$$
The $C^k (X, V)$'s are 
$G$-modules for the following action: for $g \in G$, $f \in C^k(X,V)$ and $\sigma \in X^{(k)}$, 
$$(g \cdot f) (\sigma) = \pi (g) \bigl(f(g^{-1}(\sigma)\bigl).$$
The $\delta _k$'s are $G$-morphisms.
The complex 
$$0 \to V \stackrel{\delta_{-1}}{\to} C^0(X,V) \stackrel{\delta_0}{\to} C^1(X,V)  \stackrel{\delta_1}{\to} C^2(X,V) \stackrel{\delta_2}{\to} ... $$ 
is a relatively injective strong $G$-resolution of $V$. Indeed one defines the $h_k$'s by induction on $k$ by using
a retraction of $X$ to a point. To show that $C^k (X, V)$ is relatively injective for every $k$, one considers a fundamental domain $D$
in $X$, the associated set
$$X_D ^{(k)}:= \{\sigma \in X ^{(k)} ~\vert~ \mathrm{the~first~vertex~of}~\sigma ~\mathrm{belongs~to}~D\}$$
and the vector space $C_D ^k (X,V)$ of linear maps from ${\bf R} X _D ^{(k+1)}$ to $V$. Then one can show that $C^k(X,V) \simeq 
C^0\bigl(G, C_D ^k (X,V)\bigr)$. The latter is relatively injective (see the discussion in Example \ref{bar_ex} above).
\end{example}

We now want to see that relatively injective strong $G$-resolutions compute the continuous cohomology (and, in fact, we will quote a better statement needed later). 
More precisely, if $V$ and $W$ are topological $G$-modules, if we are given a relatively injective strong $G$-resolution $0 \to V \to A^*$ and a complex of $G$-modules $0 \to W \to B^*$, then (by d\'evissage and induction) any $G$-morphism $V \to W$ can be extended to a morphism of $G$-complexes $A^* \to B^*$ \cite[Chap.\! III, \textsection 1, Prop.\! 1.1 p.\! 177]{Gui}. 

\begin{proposition}\label{resolution_prop} 
Let $V$ be a topological $G$-module. 
Assume we are given two relatively injective strong $G$-resolutions $0 \to V \to A^*$ and $0 \to V \to B^*$ of $V$. 

\noindent {\rm (i)}~The complexes of invariants 
$$K_A^* : (A^0)^G \to (A^1)^G \to (A^2)^G \to \dots$$
\noindent and 
$$K_B^* : (B^0)^G \to (B^1)^G \to (B^2)^G \to \dots$$
\noindent are homotopically equivalent. 

\noindent {\rm (ii)} If we start from the identity map ${\rm id}_V$ in order to obtain an extension $u : A^* \to B^*$, hence a map $K_A^* \to K_B^*$, then the resulting morphism ${\rm H}^*(K_A^*) \to {\rm H}^*(K_B^*)$ is a topological isomorphism which does not depend on the choice of $u$. 
\end{proposition}

This is \cite[Chap.\! III, \textsection 1, Cor.\! 1.1 p.\! 177]{Gui} and the same holds for reduced cohomology. 
Since homotopy equivalent complexes have the same cohomology, one obtains in particular:

\begin{corollary}\label{resolution_cor} Suppose $0 \to V \to A^*$ is a relatively injective strong $G$-resolution of $V$. 
Then, the cohomology and the reduced cohomology of the complex
$(A^*)^G$ are topologically isomorphic to
${\rm H}_{\mathrm{ct}} ^* (G, V)$ and $\overline{{\rm H}_{\mathrm{ct}} ^*} (G, V)$, respectively.
\end{corollary}

\begin{proof} From Example \ref{bar_ex} and Proposition \ref{resolution_prop} above, the complexes $(A^*)^G$ and $C^*(G, V)^G$
are homotopy equivalent. Therefore, by standard arguments, their cohomological spaces (equipped with the quotient topology) 
are topologically isomorphic. By definition, the cohomology of $C^*(G, V)^G$ is ${\rm H}_{\mathrm{ct}} ^* (G, V)$. 
Thus the cohomology ${\rm H}^*\bigl((A^*)^G\bigr)$ of $(A^*)^G$ is topologically isomorphic to ${\rm H}_{\mathrm{ct}} ^* (G, V)$. 
The reduced cohomological spaces of $(A^*)^G$ and $C^*(G, V)^G$ are respectively topologically isomorphic to ${\rm H}^k\bigl((A^*)^G\bigr) / \overline{\{0\}}$
and ${\rm H}_{\mathrm{ct}} ^k (G, V) / \overline{\{0\}}$, where $\overline{\{0\}}$ denotes the closure of the null subspace.
Since ${\rm H}^k\bigl((A^*)^G\bigr)$ and ${\rm H}_{\mathrm{ct}} ^k (G, V)$ are topologically isomorphic, so are ${\rm H}^k\bigl((A^*)^G\bigr) / \overline{\{0\}}$ and ${\rm H}_{\mathrm{ct}} ^k (G, V) / \overline{\{0\}}$.
\end{proof}

\section{Continuous group $L^p$-cohomology}
\label{3}

The group $L^p$-cohomology is defined in this section (Definition \ref{group_L^p_def}). When the group is a finite 
$K(\pi, 1)$, we relate its group $L^p$-cohomology with the simplicial $L^p$-cohomology (Proposition \ref{simplicial_prop}).
For locally compact second countable topological groups, the group $L^p$-cohomology is isomorphic to the asymptotic $L^p$-cohomology
(Theorem \ref{cont_asymp_thm}). This implies Theorem \ref{intro-thm1} of the introduction.

\begin{definition} 
\label{group_L^p_def} 
Let $G$ be a locally compact second countable topological group, 
and $\mathcal H$ be a left-invariant Haar measure. Let $p \in (1, +\infty)$. 
The \emph{group $L^p$-cohomology} of $G$
is the continuous cohomology of $G$, as in Definition \ref{ell_pcont_def}, with coefficients in
the right-regular representation of $G$ on $L^p (G, \mathcal H)$, \emph{i.e.\!} the representation defined by
$$\bigl(\pi (g) u\bigr)(x) = u(xg) ~~\mathrm{~~for~~}~~ u \in L^p (G, \mathcal H) ~~\mathrm{~~and~~}~~ g, x \in G.$$
It will be denoted by ${\rm H}_{\mathrm{ct}} ^* \bigl(G, L^p (G)\bigr)$.
The \emph{reduced group $L^p$-cohomology} of $G$ is defined similarly and is denoted by $\overline{{\rm H}_{\mathrm{ct}} ^*} \bigl(G, L^p (G)\bigr)$.
When $G$ is discrete we omit the subscript ``ct'' and simply write ${\rm H} ^* \bigl(G, L^p (G)\bigr)$ and $\overline{{\rm H}^*} \bigl(G, L^p (G)\bigr)$.
\end{definition}
Observe that the right regular representation is isometric if and only if $\mathcal H$ is bi-invariant.

\subsection*{The discrete group case}
\label{simplicial_subsection} 
In this paragraph we relate the group $L^p$-cohomology of 
certain discrete groups with  
the simplicial $L^p$-cohomo-\ logy of some simplicial complexes (Proposition \ref{simplicial_prop}). 
This is standard material.
We present it for completeness and also because its proof can be seen as a model for the general case.
 
Let $X$ be a simplicial complex and let $X^{(k)}$ be the set of its $k$-simplices. We will always assume that $X$ has \emph{bounded geometry}
\emph{i.e.}\! that there is an $N \in {\bf N}$ such that each simplex intersects at most $N$ simplices of arbitrary dimension (in other words, $X$ is finite-dimensional and uniformly locally finite). 

Let $C^{k,p}(X)$ be the Banach space of linear maps $u: {\bf R} X ^{(k)} \to {\bf R}$ such that
$$\Vert u \Vert ^p := \sum _{\sigma \in X^{(k)}} \vert u(\sigma) \vert ^p < \infty.$$
The \emph{simplicial $L^p$-cohomology} of $X$ is the cohomology of the complex 
$$C^{0,p}(X) \stackrel{\delta _0}{\to} C^{1,p}(X) \stackrel{\delta _1}{\to} C^{2,p}(X) \stackrel{\delta _2}{\to} ... $$
where the $\delta_k$'s are defined as in Example \ref{simplicial_ex}. 
The reduced simplicial $L^p$-cohomology is defined similarly. They are  denoted by 
$L^p {\rm H}^*(X)$ and $L^p\overline{{\rm H}^*} (X)$ respectively.

\begin{proposition}\label{simplicial_prop}
Suppose that $G$ acts by simplicial automorphisms, properly discontinuously, freely and cocompactly, 
on a locally finite contractible simplicial complex $X$. Then the complexes $C^*\bigl(G,L^p(G)\bigr)^G$ and 
$C^{*,p}(X)$ are homotopy equivalent. In particular 
there are topological isomorphisms  
${\rm H} ^* \bigl( G, L^p(G) \bigr) \simeq L^p{\rm H}^* (X)$ and 
$\overline{{\rm H} ^*} \bigl( G, L^p(G) \bigr) \simeq  L^p\overline{{\rm H}^*} (X)$.
\end{proposition} 

\begin{proof} 
According to Corollary \ref{resolution_cor} and Example \ref{simplicial_ex}, $C^*\bigl(G,L^p(G)\bigl)^G$ is homotopy equivalent to 
the complex
$$ 
C^0\bigl(X,L^p(G)\bigr)^G \stackrel{\delta_0}{\to} C^1\bigl(X,L^p(G)\bigr)^G \stackrel{\delta_1}{\to} C^2\bigl(X,L^p(G)\bigr)^G \stackrel{\delta_2}{\to} ... \ .
$$
In particular their cohomologies (reduced or not) are topologically isomorphic.
To prove the proposition it is enough to show the latter complex is topologically isomorphic to the complex $C^{*,p}(X)$.

To this end, one considers the map $\Phi: C^{k,p}(X) \to C^k\bigl(X,L^p(G)\bigr)$ defined by $\Phi (u) = f$ with
$f(\sigma)(g):= u(g\sigma)$. With the norms expressions, one sees that it is a topological embedding.
Its image is $C^k\bigl(X,L^p(G)\bigr)^G$. Indeed $\Phi (u)$ is clearly $G$-invariant. Moreover, for 
$f \in C^k\bigl(X,L^p(G)\bigr)^G$, one defines $u: X^{(k)} \to {\bf R}$ by $u(\sigma) = f(\sigma)(1)$. Since 
$f$ is invariant, one has $u (g \sigma) = f(\sigma)(g)$. Because $G$ acts properly discontinuously and 
cocompactly on $X$, by using the partition of $X^{(k)}$ into orbits, one gets that $u \in C^{k,p}(X)$.
\end{proof}


\begin{remark} \label{simplicial_remark} Equip every simplicial complex $X$ with the length metric obtained by identifying every simplex with the standard 
Euclidean one. A simplicial complex is called \emph{uniformly contractible} if it is contractible and if there exists a 
function $\rho: (0, +\infty) \to (0, +\infty)$
such that every ball $B(x,r) \subset X$ can be contracted to a point in $B\bigl(x,\rho (r)\bigr)$. 
Among bounded geometry uniformly contractible simplicial complexes, the simplicial $L^p$-cohomology (reduced or not)
is invariant under quasi-isometry. Indeed if $X$ and $Y$ are quasi-isometric 
bounded geometry uniformly contractible simplicial complexes, then the complexes $C^{*,p}(X)$ and $C^{*,p}(Y)$ are homotopy
equivalent (see \cite[p.\! 219]{G}, and \cite{BP2} for a detailed proof).
As a consequence, the above proposition implies that the group $L^p$-cohomology of fundamental groups 
of finite aspherical simplicial complexes
is invariant under quasi-isometry.   
We will give below another proof of this fact that applies in a much greater generality.
\end{remark}

\subsection*{The general case}
Suppose that $G$ is a locally compact second countable topological group. 
We will relate the group $L^p$-cohomology and the asymptotic $L^p$-cohomology of $G$.
The latter has been considered by Pansu in \cite{Pa95}, see also \cite{Genton}. It is a quasi-isometric invariant (see Theorem \ref{Pansu_thm} below). 

Following \cite{Pa95}, we define the asymptotic $L^p$-cohomology in the context of metric spaces. 
Let $(X, d)$ be a metric  space equipped with a Borel measure $\mu$. Suppose it satisfies the following ``bounded geometry''
condition. There exists increasing functions $v, V: (0, +\infty) \to (0, +\infty)$ such for every ball 
$B(x, r) \subset X$ one has 
$$v(r) \leqslant \mu \bigl(B(x,r)\bigr) \leqslant V(r).$$  
For $s >0$ and $k\in {\bf N}$, let
$\Delta _s ^{(k)} = \{(x_0, x_1, ... , x_k) \in X ^{k+1} ~\vert~  d(x_i, x_j) \leqslant s\}$.
Let $AS^{k,p}(X)$ be the set of the (classes of) functions $u: X^{k+1} \to {\bf R}$ such that for every
$s >0$ one has 
$$N_s (u) ^p:= \int _{\Delta _s ^{(k)}} \bigl\vert u (x_0, ... ,x_k) \bigr\vert ^p d\mu(x_0) ... d\mu (x_k) < \infty.$$
We equip $AS^{k,p}(X)$ with the topology induced by the set of the semi-norms $N_s$ $(s>0)$.

\begin{definition}\label{cont_asymp-def}
The asymptotic $L^p$-cohomology of $X$ is the cohomology of the complex 
$AS^{0,p}(X) \stackrel{d_0}{\to} AS^{1,p}(X) \stackrel{d_1}{\to} AS^{2,p}(X) \stackrel{d_2}{\to} ...$, 
where the $d_k$'s are defined as in (\ref{differential}).
The reduced asymptotic $L^p$-cohomology is defined similarly.
They are denoted by $L^p {\rm H}_{AS}^* (X)$
and $L^p \overline{{\rm H}_{AS}^*} (X)$ respectively.
\end{definition}
Asymptotic $L^p$-cohomology (reduced or not) is invariant under quasi-isometry. In fact one has 

\begin{theorem}[\cite{Pa95}] \label{Pansu_thm} 
Let $X$ and $Y$ be metric spaces.
We assume that each of them admits a Borel measure with respect to which it is of bounded geometry (as defined above). 
Let $F : X \to Y$ be a quasi-isometry. 
Then $F$ induces a homotopy equivalence between the complexes $AS^{*,p}(X)$ and $AS^{*,p}(Y)$. 
Moreover the associated isomorphism of graded topological vector spaces $F^* : L^p {\rm H}_{AS}^* (Y) \to L^p {\rm H}_{AS}^* (X)$ depends only on the bounded perturbation class of $F$. 
\end{theorem}

See also \cite{Genton} for a more detailed proof. In the next section we will prove the following theorem. 

\begin{theorem}\label{cont_asymp_thm} 
Suppose $G$ is a locally compact second countable topological group, equipped with a left-invariant proper metric. 
Then the complexes $C^*\bigl( G, L^p(G) \bigr)^G$ and $AS^{*,p}(G)$ are homotopy equivalent. 
In particular there exist topological isomorphisms:  
$${\rm H}_{\mathrm{ct}}^* \bigl(G, L^p (G)\bigr) \simeq L^p {\rm H}_{AS}^* (G) \mathrm{~~~and~~~} 
\overline{{\rm H}_{\mathrm{ct}} ^*} \bigl(G, L^p (G)\bigr) \simeq L^p \overline{{\rm H}_{AS}^*} (G).$$
\end{theorem}

In combination with Theorem \ref{Pansu_thm} one obtains the following result, which implies Theorem \ref{intro-thm1}
of the introduction:

\begin{corollary}
\label{cont_asymp_cor} 
Let $G_1$ and $G_2$ be locally compact second countable topological groups, equipped with left-invariant proper metrics and let $F : G_1 \to G_2$ be a quasi-isometry. 
Then $F$ induces a homotopy equivalence between the complexes 
$C^*\bigl(G_1, L^p(G_1)\bigr)^{G_1}$ and $C^*\bigl(G_2, L^p(G_2)\bigr)^{G_2}$.
Moreover the associated isomorphism of graded topological vector spaces
$F ^* :{\rm H}^*_{\mathrm {ct}} \bigl( G_2, L^p(G_2) \bigr) \to {\rm H}^*_{\mathrm {ct}} \bigl( G_1, L^p(G_1) \bigr)$  
depends only the bounded perturbation class of $F$. 
\end{corollary}

\begin{remark} Suppose $X$ is a bounded geometry simplicial complex. In general its asymptotic and simplicial
$L^p$-cohomologies are different (for example the asymptotic $L^p$-cohomology of a finite simplicial complex is trivial
since it is quasi-isometric to a point).
Pansu \cite{Pa95} asked whether the asymptotic and the simplicial $L^p$-cohomologies coincide for 
uniformly contractile bounded geometry simplicial complexes.
He proved that it is indeed the case for those which are in addition non-positively curved; for these spaces the complexes 
$C^{*,p}(X)$ and $AS^{*,p}(X)$ are homotopy equivalent.
\end{remark}

\section{Proof of quasi-isometric invariance}
\label{4}

As explained in the previous section the quasi-isometric invariance of the group $L^p$-cohomology 
is a consequence of Theorem \ref{cont_asymp_thm}. We prove this theorem in this section. 
The proof is inspired by the one of Proposition \ref{simplicial_prop}. 
It will use a relatively injective strong $G$-resolution 
(Lemma \ref{Blanc_lemma}) that appears in Blanc
\cite{Bl}.
As already mentioned in the introduction, Theorem \ref{cont_asymp_thm} was proved by G.\! Elek \cite{Elek} for finitely generated groups. In degree $1$, it was established by R.\! Tessera \cite{T}. R.~Sauer and M.~Schr\"odl obtained a result very similar to Theorem \ref{cont_asymp_thm} by using the same strategy -- see \cite[Th.\! 10]{SaSc}. They applied it to show that vanishing of $\ell^2$-Betti numbers is a coarse equivalence invariant for unimodular locally compact second countable groups.

Suppose that the topological vector space $V$ is Fr\'echet (\emph{i.e.}\! metrizable and complete).
Let $X$ be a locally compact second countable topological space endowed with a Radon measure $\mu$, \emph{i.e.}\! a Borel measure which is finite on compact subsets.
Let $p \in (1, +\infty)$. We denote by $L_{\mathrm{loc}}^p(X,V)$ the set of (classes of) measurable functions
$f: X \to V$ such that for every compact $K \subset X$ and every semi-norm $N$ on $V$ defining its topology,
one has 
$$\Vert f \Vert _{K,N}:= \Bigl( \int _K N\bigl(f(x)\bigr) ^p d\mu (x) \Bigr) ^{\frac{1}{p}} < \infty.$$
Suppose that $G$ acts on $X$ continuously by preserving $\mu$.
Equipped with the set of semi-norms $\Vert \cdot \Vert_{K,N}$, the vector space $L_{\mathrm{loc}}^p(X,V)$ is a Fr\'echet $G$-module for the action $(g \cdot f) (x) = \pi (g) \bigl(f(g^{-1} x)\bigr)$,  see \cite{Bl}. 
Let $\mathcal H$ be a left-invariant Haar measure on $G$; it is a Radon measure. Denote by $G^k$ the cartesian product of $k$ copies of $G$
equipped with the product measure. 

\begin{lemma}[\cite{Bl}]
\label{Blanc_lemma} 
The complex \\
$0 \to V \stackrel{d_{-1}}{\to} L_{\mathrm{loc}}^p(G, V) \stackrel{d_0}{\to} L_{\mathrm{loc}}^p(G^2, V) 
\stackrel{d_1}{\to} L_{\mathrm{loc}}^p(G^3, V) \stackrel{d_2}{\to} ... $, with the $d_k$'s defined as in (\ref{differential}), is a relatively injective strong $G$-resolution of $V$.
\end{lemma}
 
In the special case $V = L^p(G)$, one can express the spaces $L_{\mathrm{loc}}^p(X, V)^G$ as follows.

\begin{lemma}\label{cont_asymp_lemma} Let $X,\mu, G, \mathcal H$ as above. Let 
$E$ be the topological vector space that consists of the (classes of the) Borel functions $u: X \to {\bf R}$
such that for every compact subset $K \subset X$ one has
$$\Vert u \Vert _K ^p:= \int_G \int_K \vert u(gx) \vert^p d\mu(x)d\mathcal H (g) < \infty .$$
Then $E$ is topologically isomorphic to the Fr\'echet space $L_{\mathrm{loc}}^p\bigl(X, L^p(G)\bigr)^G$.
\end{lemma}

\begin{proof} Let $u \in E$. Since $u$ is Borel and since the map 
$$\varphi: (g,x) \in G \times X \mapsto gx \in X$$ 
is continuous,
the function $u \circ \varphi$ is Borel too, and its class depends only on the class 
of $u$. Moreover $u \circ \varphi = 0$ almost everywhere if and only if $u =0$ almost everywhere.
Indeed, given $A \subset X$, one has $\varphi ^{-1} (A)  = \cup _{g \in G} (\{g\} \times g^{-1}A)$,
and thus $(\mathcal H \times \mu) \bigl(\varphi ^{-1} (A)\bigr) =0$ if and only if $\mu (A) =0$.

Define $\Phi: E \to L_{\mathrm{loc}} ^p \bigl(X, L^p(G)\bigr)$ by $\Phi (u) = f$ with
$$f(x)(g):= u(gx) = (u \circ \varphi) (g,x) ~~\mathrm{for~~every}~~ x \in X, g \in G.$$
An element $f \in L_{\mathrm{loc}} ^p \bigl(X, L^p(G)\bigr)$  is null if and only if 
$f(x)(g) =0$ for almost all $(g,x) \in G \times X$. Therefore the above discussion in combination with the semi-norm expressions and the Fubini-Tonelli 
theorem imply that
$\Phi$ is a topological embedding.
It remains to prove that its image is the subspace $L_{\mathrm{loc}} ^p \bigl(X, L^p(G)\bigr)^G$.

Let $u \in E$, we show that $\Phi (u)$ is $G$-invariant. Write $f = \Phi (u)$ for simplicity. 
One has for $h \in G$ and for almost all $x \in X$ and $g \in G$:
$$(h \cdot f)(x)(g)= f(h^{-1}x)(gh)= u(gx) = f(x)(g).$$
Thus $f$ is $G$-invariant.

Let $f \in L_{\mathrm{loc}} ^p \bigl(X, L^p(G)\bigr)^G$, we are looking for $u \in E$ such that $\Phi (u) =f$.
One would like to define $u$ by  $u(x) = f(x)(1)$; but this has no meaning in general since $f(x) \in L^p(G)$.
We will give two constructions for the function $u$. 

\noindent {\bf First construction :} The first construction uses Lebesgue differentiation theorem. We do not know whether
this theorem is valid for all locally compact second countable topological groups. However among them, it applies 
to Lie groups \cite[Th.\! 1.8]{Hei}, and to totally discontinuous groups \cite[Th.\! 2.8.19 and 2.9.8]{Fed}.
It states that for every $\varphi \in L_{\mathrm{loc}}^1 (G)$, and almost all  $g \in G$, one has 
$$\lim _{r \to 0} \frac{1}{\mathcal H \bigl(B(g, r)\bigr)} \int_{B(g,r)} \varphi(h) d\mathcal H(h) = \varphi (g).$$
 
Let again $f \in L_{\mathrm{loc}} ^p \bigl(X, L^p(G)\bigr)^G$ be as above. Define a Borel function $u: X \to {\bf R}$ by  
$$u(x) = \limsup _{n \to \infty} \frac{1}{\mathcal H \bigl(B(1, \frac{1}{n})\bigr) } \int _{B(1, \frac{1}{n})} f(x)(h^{-1}) d\mathcal H(h).$$
Let $g \in G$. Since $f$ is $G$-invariant and $\mathcal H$ is left-invariant, one has for almost every $x \in X$ 
\begin{align*}
\int _{B(1, \frac{1}{n})} f(gx)(h^{-1}) d\mathcal H(h) &= \int _{B(1, \frac{1}{n})} f(x)(h^{-1}g) d\mathcal H(h) \\
&= \int _{B(g^{-1}, \frac{1}{n})} f(x)(h^{-1}) d\mathcal H(h).
\end{align*}
Therefore Lebesgue differentation theorem implies that $u(gx) = f(x)(g)$ for almost all $g \in G$ and $x \in X$.

\noindent {\bf Second construction :} Our second construction holds for every locally compact second countable topological group. 
It will use the following lemma which is a particular case of  \cite[Prop. \!B.5 p. \!198 Appendices]{Z}.

\begin{lemma} Let $Z$ be a locally compact second countable topological space on which $G$ acts continuously by preserving a Borel measure $\eta$. 
Suppose $f : Z \to \bf R$ is a Borel function
such that for all $g \in G$ one has $f(gz) = f(z)$ for almost all $z \in Z$.
Then there exists a $G$-invariant conull Borel subset $Z_0 \subset Z$ and a Borel $G$-invariant map $\tilde f : 
Z_0 \to \bf R$ such that $\tilde f = f$ almost everywhere.
\end{lemma} 

Let again $f \in L_{\mathrm{loc}} ^p \bigl(X, L^p(G)\bigr)^G$ be as above. By the exponential law \cite[Lem.\! 1.4]{Bl}, one has 
$$L_{\mathrm{loc}} ^p \bigl(X, L_{\mathrm{loc}} ^p(G)\bigr) \simeq L_{\mathrm{loc}} ^p (X \times G).$$
By abuse of notation we will still denote by $f : X \times G \to \bf R$ a Borel function which represents our
$f \in L_{\mathrm{loc}} ^p \bigl(X, L^p(G)\bigr)^G$. According to the above lemma, applied with $(Z, \eta)  = (X \times G, \mu \times \mathcal H)$ and 
with the action $g \cdot (x,h) = (gx, h g^{-1})$, there exists
a $G$-invariant conull Borel subset $Z_0$ and a Borel $G$-invariant map $\tilde f : 
Z_0 \to \bf R$ such that $\tilde f = f$ almost everywhere. 
Since $Z_0$ is conull and $G$-invariant, Fubini theorem implies that for every $g \in G$, the subset $\{x \in X ~\vert ~ (x,g) \in Z_0\}$ is conull and Borel. Let $X_0 = \{x \in X ~\vert ~ (x,1) \in Z_0\}$, and let $u : X \to \mathbb R$ be defined by $u(x) = \tilde f(x,1)$ when $x \in X_0$, and by $0$ 
otherwise. By construction $u$ is Borel. 

We claim that $u(gx) = f(x,g)$ for almost all $(x,g) \in X \times G$. Indeed, since $X_0$ is conull, the subset 
$\{(x,g) \in X \times G ~\vert ~ gx \in X_0 \}$
is conull too (see the argument in the first part of the proof).
Therefore for almost all $(x,g) \in X \times G$, one has $(gx, 1) \in Z_0$. Since $\tilde f$ is $G$-invariant on $Z_0$ 
and since $\tilde f = f$ a.e, one obtains for almost all $(x,g) \in X \times G$:\\
$u(gx) = \tilde f(gx, 1) = \tilde f(x,g) = f(x,g)$.
\end{proof}

According to Proposition \ref{resolution_prop} and 
Lemma \ref{Blanc_lemma}, to finish the proof of Theorem \ref{cont_asymp_thm},  it is enough to establish:

\begin{lemma}
The complexes 
$L_{\mathrm{loc}}^p\bigl(G^{*+1}, L^p(G)\bigr)^G$ and  
$AS^{*,p}(G)$ 
are topologically isomorphic.
\end{lemma}

\begin{proof}
Denote by $\mathcal H ^n$ the product measure on $G^n$.
According to Lemma \ref{cont_asymp_lemma},  
it is enough to show that the following two sets of semi-norms (on the space of measurable functions 
$u: G^{k+1} \to {\bf R}$) define the same topology. The first set of semi-norms is the one considered
in Lemma \ref{cont_asymp_lemma} in the case $X = (G ^{k+1}, \mathcal H ^{k+1})$. These are the $N_K$'s, with $K \subset G^{k+1}$ 
compact, defined by 
$$N_K (u) ^p = \int _G \int _K \bigl\vert u (gx_0, ..., gx_k) \bigr\vert ^p  d\mathcal H ^{k+2}(g, x_0, ..., x_k). $$
The second one is the set of semi-norms that appears in the definition of the asymptotic $L^p$-cohomology
of $(G, d, \mathcal H)$. These are the $N_s$'s , with $s>0$, defined by  
$$N_s (u)^p = \int _{\Delta _s } \bigl\vert u (y_0, ..., y_k)\bigr \vert ^p 
d\mathcal H ^{k+1} (y_0, ... , y_k).$$
For a compact subset $U \subset G$, set 
$$ K _s ^U:= \{(x_0, x_1, ... , x_k) \in \Delta _s ~\vert~ x_0 \in U\}.$$
Since the metric on $G$ is proper, $ K _s ^U$ is compact. Moreover every compact subset of $G^{k+1}$ is contained in such a subset.
Let $\Delta$ be the modular function on $G$. Since the metric and the mesure are left-invariant, one has 
\begin{align*}
&N _{K_s ^U} (u) ^p =\\
&=\int _{g \in G} \int _{x \in U} \int _{(1, y_1, ... , y_k) \in \Delta _s }
\bigl\vert u(gx, gxy_1, ... , gxy_k) \bigr\vert^p d\mathcal H ^{k+2}(g, x, y_1, ... ,y_k)  \\
&=\int _{g \in G} \int _{x \in U} \int _{(1, y_1, ... , y_k) \in \Delta _s } \Delta (x) 
\bigl\vert u(g, gy_1, ... , gy_k) \bigr\vert^p d\mathcal H ^{k+2}(g, x, y_1, ... ,y_k) \\
&= C(U) \cdot N_s (u) ^p,
\end{align*}
where $C(U):= \int _U \Delta (x)  d\mathcal H (x)$.
Thus the two sets of semi-norms define the same topology.
\end{proof}

\section{Semi-direct products}\label{5}

We partially relate the $L^p$-cohomology of semi-direct products $P = Q \ltimes R$  with the $L^p$-cohomology 
of the normal subgroups $R$ (see Corollaries \ref{spectral_cor1} and  \ref{spectral_cor2}). 
The next section will contain examples of application. 

\subsection*{Some generalities}
Let $V$ be a Fr\'echet vector space, and let $X$ be a locally compact second countable topological space
endowed with a Radon measure $\mu$. 
Let $p \in (1, +\infty)$.
We denote by $L^p(X, V)$ the set of (classes of) measurable functions $f: X \to V$
such that for every semi-norm $N$ on $V$ defining the topology of $V$, one has 
$$\Vert f \Vert _N ^p:= \int _X N\bigl(f(x)\bigr) ^p d\mu (x) < \infty.$$
Equipped with the set of semi-norms $\Vert \cdot \Vert _N$, the vector space $L^p(X,V)$ is a
Fr\'echet space.

\begin{proposition} \label{generalities_prop2}Suppose $X, Y$ are topological spaces as above endowed with Radon measures.
Then $L^p\bigl(X, L_{\mathrm{loc}}^p (Y, V)\bigr)$ is topologically isomorphic to $L_{\mathrm{loc}}^p \bigl(Y, L^p (X,V)\bigr)$.
\end{proposition}

\begin{proof} Let $f: X \times Y \to V$ be a measurable function. 
For a compact subset $K \subset Y$ and a continuous semi-norm $N$ on $V$, denote by $N_{K, N}$ and $\Vert \cdot \Vert _N$
the associated semi-norms on $L_{\mathrm{loc}}^p (Y, V)$ and $L^p (X,V)$ respectively. One has by Fubini-Tonelli
$$\Vert f \Vert _{N_{K, N}} ^p 
= \int _X \int _K N\bigl(f(x,y)\bigr) ^p d\mu_X (x) d\mu _Y (y)= \Vert f \Vert _{K, \Vert \cdot \Vert _N} ^p.$$
Moreover, by the exponential law \cite[Lem.\! 1.4]{Bl}, one has 
$$L_{\mathrm{loc}} ^p \bigl(X, L_{\mathrm{loc}} ^p(Y, V)\bigr) \simeq L_{\mathrm{loc}} ^p (X \times Y, V) \simeq 
L_{\mathrm{loc}} ^p \bigl(Y, L_{\mathrm{loc}} ^p(X, V)\bigr) .$$
Therefore the proposition follows.
\end{proof}

\subsection*{Hochschild-Serre spectral sequence}

Let $R$ be a closed normal subgroup of a locally compact second countable group $P$ and
let $(\pi, V)$ be a topological $P$-module.
The $P$-action on $C^*(R,V)$, defined by
$$(g \cdot f) (x_0, ..., x_k) = \pi (g)\bigl( f(g^{-1} x_0 g, ..., g^{-1} x_k g)\bigr)$$
for every $g \in P$, $f \in C^k (R,V)$ and $x_0, ..., x_k \in R$, induces a $P/R$-action on the  
topological vector spaces ${\rm H}_{\mathrm{ct}} ^*(R, V)$ \cite[p.\! 50]{Gui}. 

The Hochschild-Serre spectral sequence relates the continuous cohomology
${\rm H}_{\mathrm{ct}} ^*(P, V)$ to the cohomologies ${\rm H}_{\mathrm{ct}} ^*\bigl(P/R, {\rm H}_{\mathrm{ct}} ^*(R,V)\bigr)$ -- see \emph{e.g.}\! 
\cite[Chapter IX, Th.\! 4.1 p.\! 178]{BW} or \cite[Prop.\! 5.1 p.\! 214]{Gui}.

In the special case when $(\pi, V)$ is the right regular representation on $L^p(P)$ and $P$
is a semi-direct product $P = Q \ltimes R$ (where $Q$ and $R$ are closed subgroups)
we give in this subsection a more comprehensive description of these cohomological spaces.
In the following statement 
the measures $\mathcal H _Q$ and $\mathcal H _R$ are left-invariant Haar measures on $Q$ and $R$ respectively. 

\begin{proposition}\label{spectral_prop}  Let $k \in {\bf N}$.
Assume that ${\rm H}_{\mathrm{ct}} ^k \bigl(R, L^p(R)\bigr)$ is Hausdorff and that the complex $C^*\bigl(R, L^p(R)\bigr)^R$
is homotopically equivalent to a complex of Banach spaces. Then 
${\rm H}_{\mathrm{ct}} ^k \bigl(R, L^p (P)\bigr)$ is Hausdorff and there is a canonical isomorphism of topological $Q$-modules:
$${\rm H}_{\mathrm{ct}} ^k \bigl(R, L^p (P)\bigr) \simeq L^p \Bigl(Q, {\rm H}_{\mathrm{ct}}^k \bigl(R, L^p (R)\bigr)\Bigr),$$ 
where, in the latter space, the groups $Q$ and $R$ are equipped with the 
measures $\mathcal H _Q$ and $\mathcal H _R$, and where the $Q$-action is induced by
$$(q \cdot f)(y)( x_0, ..., x_k, x) = f(yq)(q^{-1} x_0 q, ... , q^{-1} x_k q, q^{-1} x q),$$
for every $q, y \in Q$,  $f: Q \to C^k \bigl(R, L^p(R)\bigr)$ and $x_0, ... ,x_k, x \in R$.
\end{proposition} 

\begin{proof}
As a consequence of Blanc's lemma (see Lemma \ref{Blanc_lemma}) and of Proposition \ref{resolution_prop} (ii), the inclusion of the subcomplex $C^*\bigl(R, L^p(P)\bigr)^R$ into the 
complex $L _{\mathrm{loc}} ^p \bigl(R^{*+1} , L^p (P)\bigr)^R$ induces a canonical isomorphism in cohomology.
Thus, it is not only that these two complexes have the same cohomology, but the $Q$-action on the cohomology
can be expressed in the same manner.

We claim that 
$L^p(P) \simeq L^p \bigl(Q, L^p(R)\bigr)$
as topological
$P$-modules, where $Q$ and $R$ are equipped with the measures $\mathcal H _Q$
and $\mathcal H _R$, and where
the expression of the right-regular representation of $P$ on the latter module is
$\bigl(\pi(q,r)f\bigr)(y,x) = f(yq, q^{-1}xqr)$ for every $q,y \in Q$ and $r,x \in R$.
 
Indeed this follows from Fubini, by observing that $\mathcal H _Q \times \mathcal H _R$
is a left-invariant measure on $P = Q \ltimes R$. 

Therefore, with Proposition \ref{generalities_prop2}, we obtain that $L _{\mathrm{loc}} ^p \bigl(R^{*+1} , L^p (P)\bigr)^R$
is isomorphic to $L^p \Bigl(Q, L _{\mathrm{loc}} ^p \bigl(R^{*+1} , L^p (R)\bigr)^R\Bigr)$. In this representation, one can check 
that the $Q$-action on the latter complex can be written as in the statement of the proposition.

It remains to prove that the $k$th-cohomology space of the complex $L^p \Bigl(Q, L _{\mathrm{loc}} ^p \bigl(R^{*+1} , L^p (R)\bigr)^R \Bigr)$
is isomorphic to $L^p \Bigl(Q, {\rm H}_{\mathrm{ct}} ^k \bigl(R,L^p(R)\bigr)\Bigr)$. By assumption and from Lemma \ref{Blanc_lemma}, the complex
$L _{\mathrm{loc}} ^p \bigl(R^{*+1}, L^p(R)\bigr)^R $ 
is homotopy equivalent to a complex of Banach spaces
that we denote by $B^*$. Since every continuous linear map $\varphi: V_1 \to V_2$ between
Fr\'echet spaces, extends to a continuous linear map
$\varphi _*: L^p (Q, V_1) \to L^p(Q, V_2)$ defined by $\varphi _* (f) = \varphi \circ f$,
the above homotopy equivalence induces a homotopy equivalence between
the complexes $L^p \Bigl(Q, L _{\mathrm{loc}} ^p \bigl(R^{*+1} , L^p (R)\bigr)^R\Bigr)$ 
and $L^p (Q, B^*)$.
Since ${\rm H}^k (B^*) \simeq {\rm H}_{\mathrm{ct}} ^k \bigl(R,L^p(R)\bigr)$, the following lemma ends the proof of the proposition.
\end{proof}

\begin{lemma} \label{Michael}
Suppose $k \in {\bf N}$ and that $B^*$ is a complex of Banach spaces such that
$H ^k (B^*)$ is Hausdorff. Then the $k$th-cohomology space of the complex $L^p(Q,B^*)$ satisfies ${\rm H}^k \bigl(L^p(Q, B^*)\bigr) \simeq 
L^p \bigl(Q, {\rm H}^k (B^*)\bigr)$ canonically
and topologically.
\end{lemma}

\begin{proof} [Proof of Lemma \ref{Michael}] Consider the complex $L^p(Q,B^*)$.
One has clearly $\Ker d \vert_{ L^p (Q, B^k)} = 
L^p (Q, \Ker d \vert _{B^k})$. We claim that 
$$\im \bigl(d: L^p (Q, B^{k-1}) \to L^p(Q, B^k)\bigr) = L^p \bigl(Q, \im (d: B^{k-1} \to B^k)\bigr).$$
The direct inclusion is obvious. For the reverse one, we use the assumption
that $\im (d: B^{k-1} \to B^k)$ is a Banach space in combination with the following Michael's theorem
\cite[Prop.\! 7.2]{Mi}, that previously appears in \cite[p.\! 93]{Monod_LNM} in the context of bounded cohomology:

\emph{Suppose $\varphi: B \to C$ is a continuous surjective linear map between Banach spaces.
Then for every $\lambda >1$ there exists a continuous (non linear) section $\sigma: C \to B$
such that for every $c \in C$ one has
$$\Vert \sigma (c) \Vert \leqslant \lambda \inf \{\Vert b \Vert ~;~ \varphi (b) = c\}.$$}
The natural map $L^p (Q, \Ker d \vert _{B^k}) \to L^p \bigl(Q, {\rm H}^k (B^*)\bigr)$ and the above equalities 
induce an injective continuous map 
$${\rm H}^k \bigl(L^p( Q, B^*)\bigr) \to L^p \bigl(Q, {\rm H}^k (B^*)\bigr).$$
It is surjective, thanks
again to Michael's theorem, since the projection map $\Ker d \vert _{B^k} \to {\rm H}^k (B^*)$ is a surjective 
continuous linear map between Banach spaces.

Since ${\rm H}^k (B^*)$ is a Banach space, so is $L^p \bigl(Q, {\rm H}^k (B^*)\bigr)$. From the previous description of $\Ker d \vert_{ L^p (Q, B^k)}$ and 
$\im d \vert _{ L^p (Q, B^{k-1})}$, the cohomological space ${\rm H}^k \bigl(L^p( Q, B^*)\bigr)$ is also a Banach space (for the quotient norm). Therefore Banach's theorem implies that
the above isomorphism is a topological one.
\end{proof}

As a consequence of Proposition \ref{spectral_prop}, the Hochschild-Serre spectral sequence for $L^p$-cohomology takes the following form:

\begin{corollary} \label{spectral_cor1}
Suppose that $P = Q \ltimes R$ where $Q$ and $R$ are closed subgroups of $P$.
Assume that $C^*\bigl(R, L^p(R)\bigr)^R$ is homotopically equivalent to a complex of Banach spaces and that
every cohomology space ${\rm H}_{\mathrm{ct}} ^k \bigl(R, L^p(R)\bigr)$ is Hausdorff.
Then, there exists a spectral sequence $(E_r)$, abutting to ${\rm H}_{\mathrm{ct}} ^*\bigl(P, L^p(P)\bigr)$,
in which 
$$E_2 ^{k,\ell} = {\rm H}_{\mathrm{ct}} ^k\biggl(Q, L^p\Bigl(Q, {\rm H}_{\mathrm{ct}} ^\ell \bigl(R, L^p(R)\bigr)\Bigr)\biggr).$$
\end{corollary}

\begin{corollary}\label{spectral_cor2}
Suppose that $P = Q \ltimes R$ where $Q$ and $R$ are closed subgroups of $P$.
Assume that $C^*\bigl(R, L^p(R)\bigr)^R$ is homotopically equivalent to a complex of Banach spaces. Suppose also
that there exists $n \in {\bf N}$, such that ${\rm H}_{\mathrm{ct}} ^k \bigl(R, L^p(R)\bigr) = 0$ for $0 \leqslant k < n$ and such that 
${\rm H}_{\mathrm{ct}} ^n \bigl(R, L^p(R)\bigr)$ is Hausdorff.
Then ${\rm H}_{\mathrm{ct}} ^k \bigl(P, L^p(P)\bigr) = 0$ for $0 \leqslant k < n$ and there is a linear isomorphism
$${\rm H}_{\mathrm{ct}} ^n \bigl(P, L^p(P)\bigr) \simeq L^p\Bigl(Q, {\rm H}_{\mathrm{ct}} ^n \bigl(R, L^p(R)\bigr)\Bigr) ^Q .$$
\end{corollary}
We notice that the latter isomorphism is just a linear one, in particular it does not imply that 
${\rm H}_{\mathrm{ct}} ^n \bigl(P, L^p(P)\bigr)$
is Hausdorff.

\section{The (possibly) non-vanishing strip}
\label{6}

In this section, we use the notation and conventions on Lie groups explained at the end of the introduction. 
Let $G = {\bf G}({\bf R})^\circ$ be a non-compact simple Lie group and let $A = {\bf S}({\bf R})^\circ$ be a maximal ${\bf R}$-split torus. 

Recall that the Lie algebra $\frak{g}$ of $G$ decomposes as a direct sum of $A$-stable subspaces in the adjoint representation of $G$ on $\frak{g}$ \cite[3.5]{Borel-LAG}. 
The choice of a Weyl chamber in $A$ (or in the corresponding maximal flat in the symmetric space \cite[6.4]{Maubon}) defines a system of positive roots, and the connected subgroup integrating the subspaces of $\frak{g}$ on which $A$ acts trivially or via a character which is a positive root, is a minimal parabolic subgroup in $G$.
All minimal parabolic subgroups are obtained thanks to such a choice of a Weyl chamber in a maximal ${\bf R}$-split torus $A$; moreover $G$ acts transitively on them by conjugation. 
A parabolic subgroup is a connected subgroup containing a minimal parabolic subgroup; it can also be defined via considerations from Lie algebras \cite[VII.7]{Knapp}. 
Combinatorially, the collection of parabolic subgroups of $G$ and its Weyl group come from a Tits system \cite[\textsection 21 and \textsection 24.C]{Borel-LAG}, and geometrically parabolic subgroups can be defined as stabilizers in $G$ of points in the visual boundary of the symmetric space of $G$ -- see \cite[6.6]{Maubon} and \cite{BS}. 
  
As a consequence of Corollary \ref{cont_asymp_cor}, the group $L^p$-cohomology of $G$ is isomorphic to
the group $L^p$-cohomology of any of its parabolic subgroups. Indeed, every parabolic subgroup acts cocompactly on $G$, and thus is quasi-isometric to $G$.

Recall that a parabolic subgroup decomposes as $P = M \ltimes AN$ with $M$ semi-simple, $A \simeq {\bf R} ^r$, and
$N$ nilpotent \cite[Prop. 7.83]{Knapp}. 
One can expect to use this decomposition to derive, from the Hochschild-Serre spectral sequence, some 
informations about the group $L^p$-cohomology of $P$.  To do so, and according to Corollaries \ref{spectral_cor1} and 
\ref{spectral_cor2}, 
one has to know a bit of the $L^p$-cohomology of the group $AN$. When the rank $r$ of $A$ is at least $2$, very little is known about 
the $L^p$-cohomology of $AN$ (apart from the vanishing of ${\rm H}^1 _{\mathrm{ct}} \bigl(AN, L^p(AN)\bigr)$ for all $p$ -- see \cite{P2, CT}).

In contrast, when the rank of $A$ is one, \emph{i.e.}\! when $P$ is a maximal proper parabolic subgroup, the group $AN$
admits an invariant negatively curved Riemannian metric, and the $L^p$-cohomology of negatively curved manifolds 
has been investigated by Pansu \cite{Pa99, P1}.

We focus on the class of simple Lie groups $G$ that admit a maximal proper parabolic subgroup $P = M \ltimes AN$ with $AN$
quasi-isometric to a real hyperbolic space $\mathbb H ^d$ of constant negative curvature. According to Definition \ref{intro-def} such groups $G$ are called \emph{admissible}.
They will be classified in the next section. The elementary case $G = {\rm SL}_n( {\bf R})$ is discussed in Example \ref{ex_SL}.
We obtain the following partial description of their cohomology:

\begin{theorem}
\label{theorem_vanishing}
Let $G$ be an admissible simple Lie group.
Let $P = M \ltimes AN$ be a maximal proper parabolic subgroup with radical $AN$ quasi-isometric to $\mathbb H ^d$. 
Let $X = G/K$ be the associated symmetric space and let $D = \dim X$.
One has:
\begin{enumerate}
\item ${\rm H}^0 _{\mathrm{ct}}\bigl( G, L^p(G) \bigr) = {\rm H}^k _{\mathrm{ct}}\bigl( G, L^p(G) \bigr) = 0$ for $k \geqslant D$.
\item Suppose $k \in \{1, ... ,D-1\}$. Then:
\begin{itemize}
\item ${\rm H}^k _{\mathrm{ct}}\bigl( G, L^p(G) \bigr) = 0$ for $k \leqslant \frac{d-1}{p}$ and for $k \geqslant \frac{d-1}{p} + D-d+2$,
\item $\overline{{\rm H}^k _{\mathrm{ct}}}\bigl( G, L^p(G) \bigr) =0$ for $\frac{d-1}{p} + D-d+1 \leqslant k < \frac{d-1}{p} + D-d+2$,
\item  Suppose in addition that $\frac{d-1}{p} \notin {\bf N}$, and let $\ell = \lfloor \frac{d-1}{p} \rfloor +1$. 
Then for $k > \frac{d-1}{p}$, there is a linear isomorphism:\\
${\rm H}_{\mathrm{ct}} ^k \bigl( G, L^p(G) \bigr) \simeq 
{\rm H}_{\mathrm{ct}} ^{k-\ell}  \biggl(M, L^p\Bigl(M, {\rm H}^\ell _{\mathrm{ct}} \bigl(AN, L^p(AN)\bigr)\Bigr)\biggr)$.
\end{itemize}
\end{enumerate}
\end{theorem}

In particular the possible couples $(p,k)$ for which ${\rm H}^k _{\mathrm{ct}}\bigl( G, L^p(G) \bigr) \neq 0$ lie in the strip $\frac{d-1}{p} < k <\frac{d-1}{p} + D-d+2$ of width $D-d+2$.
The numerical computation of the width will occupy Section \ref{8}.

A key ingredient in the proof of Theorem \ref{theorem_vanishing} is the following lemma which is a straightforward consequence of Pansu's results.

\begin{lemma} \label{prop_SL} Suppose that $AN$ is a Lie group quasi-isometric to a real hyperbolic space $\mathbb H ^d$.
Then for $p>1$, its group $L^p$-cohomology satisfies 
\begin{enumerate}
\item ${\rm H}^0 _{\mathrm{ct}} \bigl(AN, L^p(AN) \bigr) = {\rm H}^k _{\mathrm{ct}} \bigl(AN, L^p(AN)\bigr) = 0$ for $k \geqslant d$.
\item Suppose $k \in \{1, ... ,d-1\}$. Then:
\begin{itemize}
\item ${\rm H}^k _{\mathrm{ct}} \bigl(AN, L^p(AN)\bigr)$ is Hausdorff if and only if
$k \neq \frac{d-1}{p} +1$,
\item ${\rm H}^k _{\mathrm{ct}} \bigl(AN, L^p(AN)\bigr) = 0$ if and only if  either $k \leqslant \frac{d-1}{p}$ or
$k > \frac{d-1}{p} + 1$.
\end{itemize}
\end{enumerate}
\end{lemma}

\begin{proof}[Proof of Lemma \ref{prop_SL}]
By Corollary \ref{cont_asymp_cor}, the group $L^p$-cohomology of $AN$ is isomorphic to the group $L^p$-cohomology of any cocompact lattice in $\Isom (\mathbb {H}^d)$; 
which in turn is the same as the simplicial $L^p$-cohomology of $\mathbb H ^d$ -- see Proposition \ref{simplicial_prop}. By the simplicial $L^p$-cohomology
of a complete Riemannian manifold $X$, we mean the simplicial $L^p$-cohomology of any bounded geometry quasi-isometric simplicial decomposition of $X$. 
Now, the simplicial $L^p$-cohomology of $\mathbb H ^d$ has been computed by Pansu \cite{Pa99, P1} (\footnote{In \cite{Pa99, P1}, Pansu computes the de Rham $L^p$-cohomology
of $\mathbb H ^d$, \emph{i.e.}\! the cohomology of the de Rham complex of $L^p$ differential forms with differentials in $L^p$.
In \cite{Pa95}, he proves an $L^p$-cohomology version of the de Rham theorem, namely the homotopy equivalence between the de Rham $L^p$-complex 
and the simplicial 
$L^p$-complex, for complete Riemannian manifolds of bounded geometry. The fact that the simplicial $L^p$-cohomology is Hausdorff for $k\neq \frac{d-1}{p} +1$, and vanishes
for $k < \frac{d-1}{p}$ and  $k > \frac{d-1}{p}+1$, can also be established 
without the de Rham theorem, as a consequence of Corollary B in \cite{B16}.}).
The result of this computation is precisely the statement of our lemma.
\end{proof}

Parts of Theorem \ref{theorem_vanishing} rely on the following version of Poincar\'e duality established in
\cite{P1} Corollaire 14 (\footnote{Again, Pansu establishes this result for the de Rham $L^p$-cohomology.
One obtains the result for the simplicial $L^p$-cohomology by using the $L^p$-cohomology version of the de Rham theorem. 
The lemma can also be established more directly without the de Rham theorem, as a consequence of Proposition 1.2 and Theorem 1.3 in \cite{B16}.}).

\begin{lemma} \label{lemma_duality} Let $X$ be a complete Riemannian manifold of bounded geometry and of dimension $D$.
Denote by $L^p {\rm H}^* (X)$ and $L^p \overline{{\rm H}^*} (X)$ its simplicial $L^p$-cohomology and its reduced
simplicial $L^p$-cohomology.
Let $q = p/(p-1)$ and $0 \leqslant k \leqslant D$. Then 
\begin{enumerate}
\item $L^p {\rm H}^k (X)$ is Hausdorff if and only if $L^q {\rm H}^{D-k +1} (X)$ is Hausdorff.
\item $L^p \overline{{\rm H}^k} (X) = 0$ if and only if $L^q \overline{{\rm H}^{D-k}} (X)  = 0$.
\end{enumerate}
\end{lemma}

 \begin{proof}[Proof of Theorem \ref{theorem_vanishing}]
(1). According to Corollary \ref{cont_asymp_cor}, the group $L^p$-cohomology of $G$ is isomorphic to the group $L^p$-cohomology of any of its cocompact lattices; 
which in turn is isomorphic to the simplicial $L^p$-cohomology of $X$ -- see Proposition \ref{simplicial_prop}.  
The latter is trivially null in degree $0$ and degrees $k > \dim X$. It is also null in degree $D = \dim X$. Indeed, the lemma above implies that 
$L^p \overline{{\rm H}^D} (X) = 0$. Moreover, since $X$ satisfies a linear isoperimetric inequality, $L^p {\rm H}^1 (X)$ is Hausdorff -- see \cite{P2} or \cite{G}.
Hence the lemma implies that $L^p {\rm H}^D (X)$ is Hausdorff, and so it is null.

(2). Since $G$ and $P$ are quasi-isometric, one has ${\rm H}^* _{\mathrm{ct}}\bigl( G, L^p(G) \bigr) \simeq {\rm H}^* _{\mathrm{ct}}\bigl(P, L^p(P)\bigr)$ by Corollary \ref{cont_asymp_cor}.
The vanishing of ${\rm H}^k _{\mathrm{ct}}\bigl(P, L^p(P)\bigr)$ for every $k \leqslant \frac{d-1}{p}$, follows directly from 
the analogous result for $AN$ in Lemma \ref{prop_SL}, and from Corollary \ref{spectral_cor2} applied with $P = M \ltimes AN$, $Q = M$ and $R=AN$. 
We note that the assumptions of Corollary \ref{spectral_cor2} are satisfied; indeed $C^*\bigl(AN, L^p(AN)\bigr)^{AN}$ is homotopy equivalent to the simplicial 
$L^p$-complex of $\mathbb H ^d$  which is a complex of Banach spaces. 

To show that ${\rm H}^k _{\mathrm{ct}} \bigl( G, L^p(G) \bigr) = 0$ for $k \geqslant \frac{d-1}{p} +D -d +2$, we use again the isomorphism between the group $L^p$-cohomology 
of $G$ and the simplicial $L^p$-cohomology of 
$X$. We know from above that the latter is null for $k \leqslant \frac{d-1}{p}$.
Then Lemma \ref{lemma_duality} implies that $L^p \overline{{\rm H}^k} (X)$ is null for $k \geqslant \frac{d-1}{p} +D -d +1$ and 
that $L^p {\rm H}^k (X)$ is Hausdorff for $k \geqslant \frac{d-1}{p} +D -d +2$. 
The proof of the first two items is now complete.

Suppose that $\frac{d-1}{p}  \notin {\bf N}$. Then, according to Lemma \ref{prop_SL},
the only degree $\ell$ such that ${\rm H}^\ell _{\mathrm{ct}} \bigl(AN, L^p(AN)\bigr)$ is non-zero, is $\ell = \lfloor \frac{d-1}{p} \rfloor +1$.
Moreover ${\rm H}^\ell _{\mathrm{ct}} \bigl(AN, L^p(AN)\bigr)$ is Hausdorff. By applying Corollary \ref{spectral_cor1}, one obtains the third item.
\end{proof}

\begin{remark}\label{remark_SL}  In view of the last item in Theorem \ref{theorem_vanishing}, to analyse further the case where $\frac{d-1}{p} <k< \frac{d-1}{p} + D-d+2$, 
in particular to decide whether 
${\rm H}^k _{\mathrm{ct}} \bigl( G, L^p(G) \bigr)$ vanishes or not, one would like to take advantage of a good description of ${\rm H}^\ell _{\mathrm{ct}} \bigl(AN, L^p(AN)\bigr)$
and of the $M$-action on it. In particular 
one could try to exploit the description of ${\rm H}^\ell _{\mathrm{ct}} \bigl(AN, L^p(AN)\bigr)$ as a functional space on 
$N \simeq {\bf R}^{d-1} \simeq \partial \mathbb H ^d \setminus \{\infty \}$, that is established in Section 8.2 of
\cite{Pa99} (see also \cite{Pa1} or \cite{Rez} or \cite{BP2} for the case $\ell =1$). We will follow this idea in Section \ref{9} to prove the 
vanishing of the first group $L^p$-cohomology of admissible Lie groups of real rank $\geqslant 2$, see Corollary \ref{intro_cor2}.
\end{remark}

\begin{example}
\label{ex_SL} 
We consider here the example of the simple Lie group $G = {\rm SL}_n( {\bf R})$. 
Then the Lie algebra of $G$ is the space of square matrices of size $n$ and of trace 0, a maximal ${\bf R}$-split torus is given by diagonal matrices and it acts on the Lie algebra by conjugation. 
The corresponding root system is of type ${\rm A}_{n-1}$ and any parabolic subgroup of $G$ can be conjugated to a suitable subgroup of upper triangular-by-blocks matrices of determinant 1 \cite[11.14]{Borel-Arith}. 
Therefore, every maximal (proper) parabolic subgroup of ${\rm SL}_n( {\bf R})$ is conjugated to a subgroup with 2 blocks, i.e. of the form 
$$P = \bigg\{ 
\begin{pmatrix}
p_1 & p_3 \\
0 & p_2
\end{pmatrix}
 \in {\rm SL}_n( {\bf R}) ~\Big\vert~ p_1 \in M_{s,s} , p_2 \in M_{n-s,n-s} , p_3 \in M_{s,n-s}\bigg\},$$
with $0<s<n$. Moreover $P = M \ltimes AN$ with
$$M^0 = \bigg\{
\begin{pmatrix}
m_1 & 0 \\
0 & m_2
\end{pmatrix}
\Big\vert~ m_1 \in {\rm SL}_s( {\bf R}), m_2 \in {\rm SL}_{n-s}( {\bf R})\bigg\},$$
$$A = \bigg\{
\begin{pmatrix}
\lambda_1 I & 0 \\
0 & \lambda_2 I 
\end{pmatrix}
\Big\vert~ \lambda_1,\lambda_2 \in {\bf R} ^* _+, ~ \lambda_1 ^s \cdot \lambda_2 ^{n-s} =1 \bigg\} \simeq {\bf R} ^* _+,$$ 
and 
$N =\bigg\{
\begin{pmatrix}
I & x \\
0 & I
\end{pmatrix}
\Big\vert~ x \in M_{s, n-s} \bigg\} \simeq ({\bf R} ^{s(n-s)}, +)$.
Since one has 
$$ \begin{pmatrix}
\lambda_1 I & 0 \\
0 & \lambda_2 I 
\end{pmatrix}
\!\begin{pmatrix}
I & x \\
0 & I 
\end{pmatrix}
\!\begin{pmatrix}
\lambda_1 I & 0 \\
0 & \lambda_2 I 
\end{pmatrix}
^{-1}
= \begin{pmatrix}
I & \lambda_1 \lambda_2 ^{-1} x \\
0 & I 
\end{pmatrix}
= \begin{pmatrix}
I & \lambda_1 ^{\frac{n}{n-s}}x \\
0 & I 
\end{pmatrix},$$
we get that $AN \simeq {\bf R} \ltimes _\varphi  {\bf R} ^{s(n-s)}$, with $\varphi (t)(x) = e^{\frac{tn}{n-s}}x$.
Therefore $AN$ is isometric to $\mathbb H ^{s(n-s)+1}$ and ${\rm SL}_n( {\bf R})$ is admissible. Take $s = \lfloor \frac{n}{2} \rfloor$. Then one has
$d-1 = \dim N = s(n-s) = \frac{n^2}{4}$ if $n$ is even, and $\frac{n^2-1}{4}$ otherwise.
Since $D = \dim X = \frac{(n-1)(n+2)}{2}$, Theorem \ref{theorem_vanishing} gives the following possibly non-vanishing strip 
for the $L^p$-cohomology of ${\rm SL}_n( {\bf R})$ :
\begin{itemize}
\item $\frac{n^2}{4p} < k < \frac{n^2}{4p} + \frac{n(n+2)}{4}$ ~if $n$ is even, and
\item $\frac{n^2-1}{4p} < k < \frac{n^2-1}{4p} + \frac{(n+1)^2}{4}$ ~if $n$ is odd,
\end{itemize}
which can be summarized by:
$\lfloor \frac{n^2}{4}\rfloor \cdot\frac{1}{p} <k<\lfloor \frac{n^2}{4}\rfloor \cdot\frac{1}{p} + \lfloor \frac{(n+1)^2}{4}\rfloor $.
\end{example}

\section{The admissible simple real Lie groups}
\label{7}

We are looking for the simple Lie groups that we called admissible, \emph{i.e.} that contain a proper parabolic subgroup, say $P$, with Langlands decomposition $M_P \ltimes A^PN^P$ \cite[Prop.\! 7.83]{Knapp}, where $A^P$ is a $1$-dimensional group of simultaneously ${\bf R}$-diagonalisable matrices, where $N^P$ is a nilpotent group normalized by $A^P$, and for which the following metric condition holds: 
$(*)$ {\it the underlying Riemannian manifold of the radical $A^PN^P$ is quasi-isometric to a real hyperbolic space}.

The list of admissible simple Lie groups is given by the following (see also the tables provided at the end of Sect. \ref{8}). 

\begin{proposition}
The admissible simple real Lie groups are those whose relative root system is of type ${\rm A}_l$, ${\rm B}_l$, ${\rm C}_l$, ${\rm D}_l$, ${\rm E}_6$ and ${\rm E}_7$.
In particular, for classical types the only excluded groups are those with non-reduced root system ${\rm BC}_l$. 
\end{proposition}

\begin{proof}
We use the facts and notations recalled in the conventions of the introduction and showed in \cite[\textsection 14]{BoTi} or in \cite[\textsection 24.C]{Borel-LAG}. 
Let $P_\varnothing$ be a minimal parabolic subgroup in $G$ containing $A$, a maximal split torus. 
The subgroup $A$ defines a set $\Phi$ of roots with respect to $A$ and the inclusion $A \subset P_\varnothing$ defines, inside $\Phi$, a subset $\Phi^+$ of positive roots and a subset $\Delta$ of simple roots.
We have: $\Delta \subset \Phi^+ \subset \Phi$. 
Any parabolic subgroup is conjugated to exactly one of the parabolic subgroups containing $P_\varnothing$, and the latter subgroups are parametrized by subsets of $\Delta$ \cite[21.12]{Borel-LAG}.
More precisely, they are of the form $P_{\Delta'}= M_{\Delta'} \ltimes A^{\Delta'}N^{\Delta'}$, where $M_{\Delta'}$ is the Levi subgroup generated by the roots of $\Phi$ in the linear span ${\bf Z}\Delta'$ of $\Delta' \subset \Delta$, the subtorus $A^{\Delta'}$ of $A$ lies in the kernels of the roots in $\Delta'$ and the Lie algebra of the unipotent group $N^{\Delta'}$ is linearly spanned by the weight spaces indexed by the roots in $\Phi^+ \setminus \Phi(\Delta')$ where $\Phi(\Delta') = \Phi \cap {\bf Z}\Delta'$ \cite[21.11]{Borel-LAG}. 
This description is consistent with the notation $P_\varnothing$ for a minimal parabolic subgroup; note that we also have $G=P_\Delta$. 

Let us go back to condition $(*)$. 
By homogeneity, it implies that $A^PN^P = A^P \ltimes N^P$ is actually isometric (up to a positive multiplicative constant) to a real hyperbolic space, and in terms of roots it implies that $A^P$ must act on $N^P$ via a single character. 
The first consequence is that $P$ must be maximal for inclusion among proper parabolic subgroups.
Therefore, since we work up to conjugacy, this implies that we may -- and shall -- assume that there is a simple root $\gamma \in \Delta$ such that $P = P_{\Delta \setminus \{\gamma \}}$. 
To simplify notation, we set: $A_\gamma = A^{\Delta \setminus \{\gamma \}}$ and $N_\gamma = N^{\Delta \setminus \{\gamma \}}$, so that $P = P_{\Delta \setminus \{\gamma \}} = M_{\Delta \setminus \{\gamma \}} \ltimes A_\gamma N_\gamma$. 

The torus $A_\gamma$ is $1$-dimensional and it acts on $N_\gamma$ (more precisely, on the root subspaces of the Lie algebra of $N_\gamma$) via powers of $\gamma$ since it lies in the kernel of any other simple root of the basis $\Delta$. 
Note that the simple root $\gamma$ itself always appears as a weight in the root space decomposition of ${\rm Lie}(N_\gamma)$.
Therefore, for $A_\gamma N_\gamma$ to be homothetic to a real hyperbolic space, it is necessary and sufficient that $A_\gamma$ act on $N_\gamma$ via the single character $\gamma$ (\emph{i.e.} we must avoid powers $\gamma^k$ with $k \geqslant 2)$. 

We have obtained this way a combinatorial interpretation of condition $(*)$. 
Indeed, let $\delta$ be a root occuring as a weight in the root space decomposition of ${\rm Lie}(N_\gamma)$: it is a root in $\Phi^+$ for which the coordinate along $\gamma$ in the basis $\Delta$ has a coefficient $m \geqslant 1$ (since vanishing amounts to being in the root system $\Phi(\Delta \setminus \{ \gamma \})$ of $M_{\Delta \setminus \{\gamma \}}$). 
In this case, any $a \in A_\gamma$ acts on the root space of weight $\delta$ via the scalar $\delta(a)^m$. 

The remaining task now is to investigate the descriptions of root systems, and to check for which ones there is a simple root $\gamma \in \Delta$ with respect to which the $\gamma$-coordinate in the basis $\Delta$ of each positive root is equal to $0$ or $1$: the first case says that the positive root belongs to $\Phi(\Delta \setminus \{ \gamma \})$ and the second one says that it belongs to $\Phi^+$ hence occurs as a weight for the $A_\gamma$-action on $N_\gamma$. 
We refer now to \cite[Planches]{BBK-Lie-4to6} and its notation.
Let us investigate the root systems by a case-by-case analysis of the formulas given for the linear decompositions of the positive roots according to the given bases (formulas (II) in [loc.\! cit.]): 
\begin{enumerate}
\item[$\bullet$] for type ${\rm A}_l$, any simple root can be chosen for $\gamma$ since no coefficient $\geqslant 2$ occurs in formula (II) of [loc.\! cit.\!, Planche I]; 
\item[$\bullet$] for type ${\rm B}_l$, the root $\gamma$ can be chosen to be $\alpha_1$ in the notation of  [loc.\! cit., Planche II];
\item[$\bullet$] for type ${\rm C}_l$, the root $\gamma$ can be chosen to be $\alpha_l$ which occurs only for two kinds of positive roots, each time with coefficient 1, in formula (II) of [loc.\! cit., Planche III];
\item[$\bullet$] for type ${\rm D}_l$, we use the formulas in [loc.\! cit., p.208] instead of Planche IV (where the formulas are not correct), to see that the root $\gamma$ can be chosen to be $\alpha_1$,  $\alpha_{l-1}$ or $\alpha_l$;
\item[$\bullet$] for type ${\rm E}_6$, the root $\gamma$ can be chosen to be $\alpha_1$ or $\alpha_6$, as inspection of the positive roots with some coefficient $\geqslant 2$ shows in [loc.\! cit., Planche V];
\item[$\bullet$] for type ${\rm E}_7$, the root $\gamma$ can be chosen to be $\alpha_7$, as inspection of the positive roots whose support contains $\alpha_7$ and with some coefficient $\geqslant 2$ shows in [loc.\! cit., Planche VI].
\end{enumerate}

To exclude the remaining case, it is enough to note that:
\begin{enumerate}
\item[$\bullet$] for type ${\rm E}_8$, the last positive root described on the first page of [loc.\! cit., Planche VII] has all its coefficients $\geqslant 2$;
\item[$\bullet$] for type ${\rm F}_4$ (resp. ${\rm G}_2$), the last root in (II) of [loc.\! cit., Planche VIII (resp. IX)] has the same property; 
\item[$\bullet$] for the only non-reduced irreducible root system of rank $l$, namely ${\rm BC}_l$, the positive root $\varepsilon_1$ has all its coefficients $\geqslant 2$ [loc.\! cit., p. 222].
\end{enumerate}
This finishes the determinations of the admissible simple Lie groups by means of their relative root systems. 
\end{proof}

As a complement, we can say that the groups that are not admissible can be listed thanks to \'E.~Cartan's classification as stated for instance in \cite[Table VI, pp.\! 532-534]{Helgason}.
The classical groups with non-reduced relative root system correspond to the cases where the value of $m_{2\lambda}$ is $\geqslant 1$ in this table. 
To sum up, we can reformulate the previous proposition in more concrete terms: 

\bigskip
\begin{proposition}
The admissible simple real Lie groups are the split simple real Lie groups of classical types and of types ${\rm E}_6$ and ${\rm E}_7$, the complex simple groups of classical types and of types ${\rm E}_6$ and ${\rm E}_7$ (all seen as real Lie groups), all non-compact orthogonal groups ${\rm SO}_{p,q}({\bf R})$, all special linear groups ${\rm SL}_n({\bf H}) = {\rm SU}^*_{2n}({\bf R})$ over the quaternions ${\bf H}$, the special unitary groups ${\rm SU}_{n,n}({\bf R})$ with $n \geqslant 2$, the groups ${\rm Sp}_{2n,2n}({\bf R}) = {\rm SU}_{n,n}({\bf H})$ with $n \geqslant 2$, the groups ${\rm SO}^*_{4n}({\bf R})={\rm SO}_{2n} ({\bf H})$ with $n \geqslant 2$, the exceptional groups of absolute type ${\rm E}_6$ and real rank $2$ and of absolute type ${\rm E}_7$ and real rank $3$.
\end{proposition}

For all pratical purposes, we also provide the list of non-admissible simple real Lie groups. These are:
\begin{itemize}
\item all the compact simple groups,
\item the groups ${\rm SU}_{p,q}({\bf R})$ and ${\rm Sp}_{p,q}({\bf R})$ with $0<p<q$,
\item the simple groups ${\rm SO}^*_{2r}({\bf R})$ with odd $r$, 
\item the split groups, or the complex ones seen as real groups, of type ${\rm E}_8$, ${\rm F}_4$, ${\rm G}_2$, 
\item for each of ${\rm E}_6$, ${\rm E}_7$, ${\rm E}_8$, the real form of rank $4$,
\item the outer real form of rank $2$ of ${\rm E}_6$,
\item the real form of rank $1$ of ${\rm F}_4$.
\end{itemize}

\section{Numerical efficiency of the method}
\label{8}

Let us turn now to the numerical efficiency of the method, namely the computation of the actual width of the (possibly) non-vanishing strips, \emph{i.e.\!} the intervals outside of which our method proves vanishing of continuous group $L^p$-cohomology (Theorem \ref{theorem_vanishing}).

Recall that for a given admissible group $G$ with maximal compact subgroup $K$, this width is equal to $D-d+2$ where $D$ is the dimension of the symmetric space $X=G/K$ and $d$ is the dimension of the solvable radical $A_\gamma N_\gamma$. 
We keep the notation (in particular the notation $\gamma$) of the previous section. 
The simple root $\gamma$ defines a maximal proper parabolic subgroup, whose solvable radical contains a $1$-dimensional ${\bf R}$-diagonalizable part $A_\gamma$. 
Therefore we have $d = 1 + {\rm dim} \, N_\gamma$ and in fact: 
$$d-1 = {\rm dim} \, N_\gamma = \sum_{\alpha \in \Phi^+ \setminus \Phi(\Delta \setminus \{ \gamma \})}m_\alpha,$$
where $m_\alpha$ is the multiplicity ${\rm dim} \, N(\alpha)$ of the root space $N(\alpha)$ attached to $\alpha$.
Now, we have just seen that the root system of an admissible group is reduced, so we can freely use \cite{Humphreys-LA} (in which root systems are assumed to be reduced by definition). 
By [loc.\! cit., Lemma \! C p.\! 53], there are at most two root lengths in $\Phi$ and the Weyl group acts transitively on the roots of given length, therefore since this Weyl group action lifts to the action of $N_G(A)$ on the root spaces $N(\alpha)$, we deduce that there are at most two multiplicities, say $m_1$ and $m_2$, and with the notation $\Psi = \Phi^+ \setminus \Phi(\Delta \setminus \{ \gamma \})$ the previous formula becomes: 
$$d-1 = 
m_1 \cdot | \{\alpha \in \Psi : m_\alpha = m_1 \}| 
+ 
m_2 \cdot | \{\alpha \in \Psi : m_\alpha = m_2 \}|.$$

The case of split groups is easy since all multiplicities $m_\alpha$ are equal to 1 (hence in this case $d=1 + |\Psi|$), and the case of Weil restrictions (\emph{i.e.}\! complex groups seen as real ones) is easy too: all multiplicities are equal to 2, hence in this case $d=1 + 2 \cdot |\Psi|$. 

For the other cases, where two multiplicities may occur, we proceed as follows. 
For each admissible group, thanks to the previous section we make the choice (often, but not systematically, unique) of a {\it good}~root, that is a simple root $\gamma$ such that $A_\gamma N_\gamma$ is quasi-isometric to a real hyperbolic space. 
Thanks to \cite[Planches]{BBK-Lie-4to6}, this gives easily $|\Psi|$ since the root system of $\Phi(\Delta \setminus \{ \gamma \})$ is given by the Dynkin diagram of $\Phi$ minus the vertex of type $\gamma$ and the edges emanating from it. 
We must then sort the roots of $\Psi$ according to their length, which is done again thanks to the concrete descriptions of [loc. cit.]. 
It remains then to apply the last formula for $d-1$ and to use the multiplicities given in \cite[Table VI, pp.\! 532-534]{Helgason}. 

\smallskip

The rest of this section is dedicated to computing the width $d = D-d+2$ of the (possibly) non-vanishing strip, and to determine the minimal one when the choice of a good root $\gamma$ is not unique for the group $G$ under consideration. 

\smallskip

{\it The case of split groups}.---~Investigating the class of admissible split simple real Lie groups is the opportunity to provide a description of the roots sets $\Psi$ of nilpotent radicals, that will be useful in the more complicated cases; we use the notation of \cite[Planches]{BBK-Lie-4to6}. 

\begin{enumerate} 
\item[$\bullet$] When $G = {\rm SL}_{n+1}({\bf R})$, the root system has type ${\rm A}_n$, any simple root $\alpha_i$ $(1 \leqslant i \leqslant n)$ is a good root, and we have $|\Psi| = i \times (n-i)$; all roots have the same length. 
\item[$\bullet$] When $G = {\rm SO}_{n+1,n}({\bf R})$, the root system has type ${\rm B}_n$, the only good root is $\alpha_1$ and we have: $|\Psi| = |\Phi^+({\rm B}_n)|-|\Phi^+({\rm B}_{n-1})|=n^2 -(n-1)^2=2n-1$, with one root of length 1 and $(n-1)+(n-1)=2n-2$ roots of length $\sqrt{2}$. 
\item[$\bullet$] When $G = {\rm Sp}_{2n}({\bf R})$, the root system has type ${\rm C}_n$, the only good root is $\alpha_n$ and we have: $|\Psi| = |\Phi^+({\rm C}_n)|-|\Phi^+({\rm A}_{n-1})|=n^2 -{n(n-1)\over 2}={n(n+1)\over 2}$, with ${n(n-1)\over 2}$ roots of length $\sqrt{2}$ and $n$ roots of length $2$. 
\item[$\bullet$] When $G = {\rm SO}_{n,n}({\bf R})$, the root system has type ${\rm D}_n$, the good roots are $\alpha_1$, $\alpha_{n-1}$ and $\alpha_n$; all roots have the same length. 
For the choice $\gamma=\alpha_1$, we have: $|\Psi| = |\Phi^+({\rm D}_n)|-|\Phi^+({\rm D}_{n-1})|=n(n-1)-(n-1)(n-2)=2(n-1)$ and for $\gamma=\alpha_{n-1}$ or $\alpha_n$, we have:
$|\Psi| = |\Phi^+({\rm D}_n)|-|\Phi^+({\rm A}_{n-1})| = n(n-1) - {n(n-1) \over 2} = {n(n-1) \over 2}$.
\item[$\bullet$] When $G = {\rm E}_6^6({\bf R})$, the root system has type ${\rm E}_6$, the good roots are $\alpha_1$ and $\alpha_6$; all roots have the same length. 
Whatever the choice, we have: 
$|\Psi| = |\Phi^+({\rm E}_6)|-|\Phi^+({\rm D}_5)| = 36-20 = 16$.
\item[$\bullet$] When $G = {\rm E}_7^7({\bf R})$, the root system has type ${\rm E}_7$, the only good root is $\alpha_7$; all roots have the same length and we have 
$|\Psi| = |\Phi^+({\rm E}_7)|-|\Phi^+({\rm E}_6)| = 63-36 = 27$.
\end{enumerate} 

As already mentioned, in this case we can then compute $d=1+|\Psi|$ and take $w=D-d+2$ to be minimal when several good roots are available. 

{\it The case of complex groups seen as real ones}.---~Then the multiplicities are 2; in this case we have $d=1+2|\Psi|$ and again we can take $w=D-d+2$ to be minimal when several good roots are available. 

{\it The remaining admissible simple real Lie groups}.---~
We use here the formula $d-1 = 
m_1 \cdot | \{\alpha \in \Psi : m_\alpha = m_1 \}| 
+ 
m_2 \cdot | \{\alpha \in \Psi : m_\alpha = m_2 \}|$ and the above partition of $\Psi$ into short and long roots when the root system $\Psi$ is not simply laced, combined with the multiplicities given in \cite[Table VI, pp.\! 532-534]{Helgason}. 

\begin{enumerate} 
\item[$\bullet$] When $G = {\rm SL}_{n+1}({\bf H})$, the root system has type ${\rm A}_n$, any simple root $\alpha_i$ $(1 \leqslant i \leqslant n)$ is good, we have $|\Psi| = i \times (n-i)$ and all roots have multiplicity 4, so that $d-1 = 4i(n-i)$ for $\gamma=\alpha_i$. 
\item[$\bullet$] When $G = {\rm SU}_{n,n}({\bf R})$, the relative root system has type ${\rm C}_n$, the good root is $\alpha_n$, the ${n(n-1)\over 2}$ roots of length $\sqrt{2}$ have multiplicity 2 
and the $n$ roots of length $2$ have multiplicity 1, so that $d-1 = n(n-1)+n=n^2$. 
\item[$\bullet$] When $G = {\rm Sp}_{2n,2n}({\bf R})$, the relative root system has type ${\rm C}_n$, the good root is $\alpha_n$, the ${n(n-1)\over 2}$ roots of length $\sqrt{2}$ have multiplicity 4
and the $n$ roots of length $2$ have multiplicity 3, so that $d-1 = 2n(n-1)+3n = 2n^2+n$. 
\item[$\bullet$] When $G = {\rm SO}^*_{4n}({\bf R})$, the relative root system has type ${\rm C}_n$, the good root is $\alpha_n$, the ${n(n-1)\over 2}$ roots of length $\sqrt{2}$ have multiplicity 4
and the $n$ roots of length $2$ have multiplicity 1, so that $d-1 = 2n(n-1)+n = 2n^2-n$. 
\item[$\bullet$] When $G = {\rm E}_6^2({\bf R})$, the relative root system has type ${\rm A}_2$, any of two simple roots is good and all roots have multiplicity 8, so that $d-1=16$. 
\item[$\bullet$] When $G = {\rm E}_7^3({\bf R})$, the relative root system has type ${\rm C}_3$, the good root is $\alpha_3$, there are 3 roots of each length $\sqrt{2}$ and $2$, and since the multiplicities are 8 and 1, this gives 
$d-1=27$. 

\end{enumerate} 

At last, in the case of non-split orthogonal groups ${\rm SO}_{p,q}({\bf R})$ (\emph{i.e.\!}, where $q-2 \geqslant p \geqslant 2$), we always obtain a relative root system of type $B_p$, with $2p-2$ roots of length $\sqrt{2}$ and multiplicity 1 and one root of length 1 and multiplicity $q-p$, so that $d-1=p+q-2$.

The next two pages contain tables which summarize the vanishing results obtained by our method. 
In conclusion, apart from the general family of non-split orthogonal groups (which is special since it depends on two parameters), the infinite families of classical admissible groups provide a vanishing proportion of ${1 \over 2}$ asymptotically in the rank. 
For non-split orthogonal groups, fixing the rank (\emph{i.e.\!} the smallest number in the signature) and letting the bigger parameter go to infinity provide a vanishing proportion equal to the inverse of the rank. 

\vfill\eject

\begin{turn}{90}
\begin{tabular}{| M{2cm} || M{1.2cm} | M{1.7cm} | M{1.3cm} | M{2cm} | c | M{2.7cm} | M{2.1cm} | M{2.8cm} |}
\hline

Admissible group & 
Cartan type & 
Relative root system $\Phi$ & 
Good root $\gamma$ & 
$ |\Psi |$ & 
Multiplicities & 
$d-1 = {\rm dim} \, N_\gamma$ & 
$D = {\rm dim} \, X$ 
& Asymptotic vanishing proportion $\displaystyle \lim_{l \to +\infty} \scriptstyle {d(l) \over D(l)}$ \tabularnewline
\hline\hline

${\rm SL}_{l+1}({\bf R})$ &
A I &
${\rm A}_l$ &
any simple root $\alpha_i$ &
$i(l+1-i)$ & 
$1$ &
$i(l+1-i)$ best choice: ${(l+1)^2 \over 4}$ (odd $l$), ${l(l+2) \over 4}$ (even $l$) & 
$\displaystyle {l(l+3) \over 2}$ & 
${1 \over 2}$ \tabularnewline
\hline

${\rm SL}_{l+1}({\bf H})$ &
A II &
${\rm A}_l$ &
any simple root $\alpha_i$ &
$i(l+1-i)$ & 
$4$ &
$4i(l+1-i)$ best choice: $\scriptstyle (l+1)^2$ (odd $l$), $\scriptstyle l(l+2)$ (even $l$) & 
$\displaystyle l(2l+3)$ & 
${1 \over 2}$ \tabularnewline
\hline

${\rm SU}_{l,l}({\bf R})$ &
A III &
${\rm C}_l$ &
$\alpha_l$ &
${l(l+1) \over 2}$ & 
$2$ and $1$ &
$l^2$ & 
$2l^2$ & 
${1 \over 2}$ \tabularnewline
\hline

${\rm Sp}_{2l}({\bf R})$ &
C I &
${\rm C}_l$ &
$\alpha_l$ &
${l(l+1) \over 2}$ & 
$1$ &
$\displaystyle{l(l+1) \over 2}$ & 
$l(l+1)$ & 
${1 \over 2}$ \tabularnewline
\hline

${\rm Sp}_{2l,2l}({\bf R})$ &
C II &
${\rm C}_l$ &
$\alpha_l$ &
${l(l+1) \over 2}$ & 
$4$ and $3$ &
$2l^2+l$ & 
$4l^2$ & 
${1 \over 2}$ \tabularnewline
\hline

${\rm SO}_{l,l}({\bf R})$ &
D I &
${\rm C}_l$ &
$\alpha_1$, $\alpha_{l-1}$ and $\alpha_l$ &
$2(l-1)$ for $\gamma = \alpha_1$, ${l(l-1) \over 2}$ otherwise & 
$1$ &
best choice: ${l(l-1) \over 2}$ & 
$l^2$ & 
${1 \over 2}$ \tabularnewline
\hline

${\rm SO}^*_{4l}({\bf R})$ &
D III &
${\rm C}_l$ &
$\alpha_l$ &
${l(l+1) \over 2}$ & 
$4$ and $1$ &
$2l^2-l$ & 
$2l(2l-1)$ & 
${1 \over 2}$ \tabularnewline
\hline

${\rm E}_6^6({\bf R})$ &
E I &
${\rm E}_6$ &
$\alpha_1$ and $\alpha_6$ &
$16$ & 
$1$ &
$16$ & 
$42$ & 
${5 \over 14}$ \tabularnewline
\hline

${\rm E}_6^2({\bf R})$ &
E IV &
${\rm A}_2$ &
$\alpha_1$ and $\alpha_2$ &
$2$ & 
$8$ &
$16$ & 
$26$ & 
${15 \over 26}$ \tabularnewline
\hline

${\rm E}_7^7({\bf R})$ &
E V &
${\rm E}_7$ &
$\alpha_7$ &
$27$ & 
$1$ &
$27$ & 
$70$ & 
${13 \over 35}$ \tabularnewline
\hline

${\rm E}_7^3({\bf R})$ &
E VII &
${\rm C}_3$ &
$\alpha_3$ &
$6$ & 
$8$ and $1$ &
$27$ & 
$54$ & 
${13 \over 27}$ \tabularnewline
\hline
\end{tabular}
\end{turn}

\bigskip

\begin{turn}{90}
\begin{tabular}{| M{2cm} || M{1.2cm} | M{1.7cm} | M{1.3cm} | M{2cm} | c | M{2.8cm} | M{2.1cm} | M{2.8cm} |}
\hline

Admissible group & 
Cartan type & 
Relative root system & 
Good root $\gamma$ & 
$|\Psi |$ & 
Multiplicities & 
$d-1 = {\rm dim} \, N_\gamma$ & 
$D = {\rm dim} \, X$ 
& Asymptotic vanishing proportion $\displaystyle \lim_{l \to +\infty} \scriptstyle {d(l) \over D(l)}$ \tabularnewline
\hline\hline

${\rm SL}_{l+1}({\bf C})$ &
-- &
${\rm A}_l$ &
any simple root $\alpha_i$ &
$i(l+1-i)$ & 
$2$ &
$2i(l+1-i)$ best choice: ${(l+1)^2 \over 2}$ (odd $l$), ${l(l+2) \over 2}$ (even $l$)
& 
$\displaystyle l(l+2)$ & 
${1 \over 2}$ \tabularnewline
\hline

${\rm SO}_{l,l+1}({\bf C})$ &
-- &
${\rm B}_l$ &
$\alpha_1$ &
$2l-1$ & 
$2$ &
$4l-2$ & 
$l(2l+1)$ & 
$0$ \tabularnewline
\hline

${\rm Sp}_{2l}({\bf C})$ &
-- &
${\rm C}_l$ &
$\alpha_l$ &
${l(l+1) \over 2}$ & 
$2$ &
$l(l+1)$ & 
$l(2l+1)$ & 
${1 \over 2}$ \tabularnewline
\hline

${\rm SO}_{l,l}({\bf C})$ &
-- &
${\rm D}_l$ &
$\alpha_1$, $\alpha_{l-1}$ and $\alpha_l$ &
$2(l-1)$ for $\gamma = \alpha_1$, ${l(l-1) \over 2}$ otherwise & 
$2$ &
best choice: $l(l-1)$ & 
$l(2l-1)$ & 
${1 \over 2}$ \tabularnewline
\hline

${\rm E}_6({\bf C})$ &
-- &
${\rm E}_6$ &
$\alpha_1$ and $\alpha_6$ &
$16$ & 
$2$ &
$32$ & 
$78$ & 
${31 \over 78}$ \tabularnewline
\hline

${\rm E}_7({\bf C})$ &
-- &
${\rm E}_7$ &
$\alpha_7$ &
$27$ & 
$2$ &
$54$ & 
$133$ & 
${53 \over 133}$ \tabularnewline
\hline
\end{tabular}
\end{turn}

\section{Vanishing in degree one}\label{9}
As a further application of Theorem \ref{theorem_vanishing}, we give in this section a proof of 
Corollary \ref{intro_cor2}, namely vanishing of ${\rm H}^1 _{\mathrm {ct}} \bigl( G, L^p(G) \bigr)$ for every $p>1$ and every admissible 
simple Lie group $G$ of real rank $\geqslant 2$.

Let $G$ be an admissible simple Lie group, and let $P = M \ltimes AN$
be a maximal parabolic subgroup with $AN$ quasi-isometric to $\mathbb H ^d$.
One has ${\rm H}_{\mathrm{ct}} ^1 \bigl(G, L^p (G)\bigr) = \{0\}$ for $p \leqslant d-1$, thanks to Theorem \ref{theorem_vanishing}.
Suppose that $p > d-1$. Then Theorem \ref{theorem_vanishing} again implies that ${\rm H}_{\mathrm{ct}} ^1 \bigl(G, L^p (G)\bigr)$ is linearly
isomorphic to 
$${\rm H}_{\mathrm{ct}} ^0 \biggl(M, L^p \Bigl(M, {\rm H}_{\mathrm{ct}} ^1 \bigl(AN, L^p (AN)\bigr)\Bigr)\biggr) \simeq 
L^p \Bigl(M, {\rm H}_{\mathrm{ct}} ^1 \bigl(AN, L^p (AN)\bigr)\Bigr) ^M .$$
By \cite{Pa1} (see also \cite{Rez} or \cite{BP2}) the topological vector space ${\rm H}_{\mathrm{ct}} ^1 \bigl(AN, L^p (AN)\bigr)$
is canonically isomorphic to the following Besov space on 
${\bf R}^{d-1} \simeq N \simeq \partial \mathbb H ^d \setminus \{\infty \}$:
$$B^{(d-1)/p} _{p,p} ({\bf R} ^{d-1}) :=
\{u : {\bf R} ^{d-1} \to {\bf R} ~\mathrm{measurable}
~;~ \Vert u \Vert _B < +\infty \} / {\bf R}$$
$$\mathrm{where}~~~~~\Vert u \Vert _B ^p = \int _{{\bf R} ^{d-1} \times {\bf R} ^{d-1}}
\frac{\vert u(x) - u(x') \vert ^p}{\Vert x-x' \Vert ^{2(d-1)}} dxdx'$$
and ''$/{\bf R}$'' means dividing by the constant functions. 
The $M$-action on $L^p\Bigl(M, {\rm H}^1 _{\mathrm{ct}} \bigl(AN, L^p(AN)\bigr)\Bigr)$ described in Proposition \ref{spectral_prop}, 
takes now the following form. Denote by $(m, x) \mapsto m(x)$ the linear action of $M$ by conjugation on $N$.
Then for  $m, y  \in M$,  
$f : M \to B^{(d-1)/p} _{p,p} ({\bf R} ^{d-1})$ and $x \in N \simeq {\bf R}^{d-1}$, one has 
$$(m \cdot f) (y)(x) = f(ym)\bigl(m ^{-1} (x)\bigr).$$
Therefore ${\rm H}_{\mathrm{ct}} ^1 \bigl(G, L^p (G)\bigr)$ is linearly isomorphic to the fixed point space
$L^p \bigl(M , B_{p,p} ^{(d-1)/p} ({\bf R} ^{d-1}) \bigr)^M$, which in turn is isomorphic to 
$$\Bigl\{u \in B_{p,p} ^{(d-1)/p} ({\bf R} ^{d-1})~\Big\vert \int _M \!\Vert u \circ m \Vert _B ^p d\mathcal H (m) 
<  \infty \Bigr\}.$$
Since $M$ acts on ${\bf R} ^{d-1}$ linearly and preserves the volume, a change of variable and Fubini-Tonelli yield
\begin{equation}\label{H^1_eq0}
\int _M \Vert u \circ m \Vert _B ^p d\mathcal H (m) 
= \int _{{\bf R} ^{d-1} \times {\bf R}^{d-1}} \frac{\vert u(x) - u(x') \vert ^p}{\Vert x-x' \Vert ^{2(d-1)}}
\varphi (x-x') dxdx',
\end{equation}
where $\varphi$ is the following positive function on ${\bf R} ^{d-1} \setminus \{0 \}$:
\begin{equation}\label{H^1_eq}
\varphi (v) = \int _M \frac{d \mathcal H (m)}{\bigl\Vert m^{-1} (\frac{v}{\Vert v \Vert}) \bigr\Vert  ^{2(d-1)}} .
\end{equation}

We claim that 
$\int _M \Vert u \circ m \Vert _B ^p d\mathcal H (m) = +\infty$ unless $u$ is constant \emph{a.e.}; the corollary will follow. 
The proof relies on three lemmata. The first one gives a sufficient condition for $u$ to be constant.

\begin{lemma}\label{H^1_lem0}
Let $u \in B_{p,p} ^{(d-1)/p} ({\bf R} ^{d-1})$ that is invariant under a $1$-parameter group of translations.
Then $u$ is constant \emph{a.e.}
\end{lemma}

\begin{proof}
Express the Besov norm of $u$ as
$\Vert u \Vert _B ^p = \int _{{\bf R} ^{d-1} \times {\bf R} ^{d-1}} \psi~dxdx'$ with 
$$\psi (x,x') = \frac{\vert u(x) - u(x') \vert ^p}{\Vert x - x' \Vert ^{2(d-1)}}.$$
Suppose that $u$ is invariant under a $1$-parameter group of translations along some vector line ${\bf R}v_0$. Then $\psi$ is invariant under the $1$-parameter group of translations along ${\bf R}(v_0, v_0)$. If in addition one has
$\int _{{\bf R} ^{d-1} \times {\bf R} ^{d-1}} \psi~ dxdx' < +\infty$, then Fubini-Tonelli implies that $\psi$ is null \emph{a.e.}
This in turn implies that $u$ is constant \emph{a.e.}
\end{proof} 

The second lemma collects some simple properties of the function $\varphi$ defined in (\ref{H^1_eq}).

\begin{lemma}\label{H^1_lem1}
The function $\varphi$ admits the following properties:
\begin{enumerate}
\item $\varphi (\lambda v) = \varphi (v)$ for every $\lambda \in {\bf R}^*$
and every $v \in {\bf R} ^{d-1} \setminus \{0 \}$.
\item There exists a positive constant $c_0$ such that $\varphi \geqslant c_0$.
\item The function $v \mapsto \frac{\varphi (v)}{\Vert v \Vert ^{2(d-1)}}$ is $M$-invariant on ${\bf R} ^{d-1} \setminus \{0 \}$.
\item If $v \in {\bf R} ^{d-1} \setminus \{0 \}$, $g \in M$ and $\lambda \in {\bf R}$, are such that $g(v) = \lambda v$ with $\vert \lambda \vert \neq 1$,
then $\varphi (v) = +\infty$.
\end{enumerate}
\end{lemma}

\begin{proof} To prove the second item, observe that for every relatively compact open subset $U \subset G$, the map
$$v \in {\bf R} ^{d-1} \setminus \{0 \} \longmapsto \int _U \frac{d \mathcal H (m)}{\bigl\Vert m^{-1} (\frac{v}{\Vert v \Vert}) \bigr\Vert  ^{2(d-1)}}$$
is positive, continuous, invariant by scalar multiplication, and smaller that $\varphi$.
The third item follows from the fact that $\mathcal H$ is a left-invariant measure.
To obtain the last item, one observes that the first and third items imply that 
$$\varphi (v) = \varphi (\lambda v) = \varphi (g(v)) = \frac{\Vert g(v) \Vert ^{2(d-1)}}{\Vert v \Vert ^{2(d-1)}} \varphi (v)
= \vert \lambda \vert ^{2(d-1)} \varphi (v).$$
Thus $\varphi (v) = +\infty$ by (2). 
\end{proof} 

We now describe the behaviour of $\varphi$ in a neighborhood of an eigenvector.
\begin{lemma}\label{H^1_lem2}
Let $(g^t)_{t\in {\bf R}} \subset M$ be a non-trivial $1$-parameter subgroup which acts diagonally with positive eigenvalues on 
${\bf R} ^{d-1}$. Let $\lambda _{\mathrm {max}} >1$ and $\lambda _{\mathrm {min}}<1$ be the maximum and minimum eigenvalues of $g$. 
Let $v_0 \in {\bf R} ^{n-1}$ be an eigenvector of $g$ for the eigenvalue $\lambda _{\mathrm {max}}$. 
There exist a constant $c_2 > 0$ and a neighborhood $U$ of $v_0$
such that for every $v \in U$ one has 
$$\varphi (v) \geqslant \frac{c_2}{\Vert v-v_0 \Vert ^\alpha}$$ 
with $\alpha = 2(d-1) \frac{\log \lambda _{\mathrm {max}}}{\log \lambda _{\mathrm {max}} - \log \lambda _{\mathrm {min}}}$.
\end{lemma}

\begin{proof} If $v$ is an eigenvector to the eigenvalue $\lambda _{\mathrm {max}}$, then 
$\varphi (v) = +\infty$ thanks to Lemma \ref{H^1_lem1}. Suppose now that $v$ is not such an eigenvector. 
Let $W \subset {\bf R} ^{d-1}$ be a subspace such that 
${\bf R} ^{d-1} = {\bf R} v_0 \oplus W$ is a $\{g^t\}_{t\in {\bf R}}$-invariant decomposition. We write every vector $v \in {\bf R}^{d-1}$ 
as  $v = v_1 + v'$ according to this decomposition. For simplicity we will abusively denote by the same symbol 
$\Vert \cdot \Vert$ the standard norm on ${\bf R} ^{d-1}$ and the norm $\max (\Vert v_1 \Vert, \Vert v' \Vert)$.

Suppose that $\Vert v' \Vert \leqslant \Vert v_1 \Vert$. Then there exists $t \geqslant 0$ such that 
$\Vert g^{-t} (v') \Vert = \Vert g^{-t} (v_1) \Vert$. Write $w = g^{-t} (v) = g^{-t}(v_1) + g^{-t}(v') = w_1 + w'$. It follows from the previous lemma that
\begin{align*}
\varphi (v) = \varphi (g^t (w)) & = \frac{\Vert g^t(w) \Vert ^{2(d-1)}}{\Vert w \Vert ^{2(d-1)}} \varphi (w)
= \frac{\Vert \lambda ^t _{\mathrm {max}} w_1 + g^t(w') \Vert ^{2(d-1)}}{\Vert w_1 + w' \Vert ^{2(d-1)}} \varphi (w)\\
&\geqslant \frac{(\lambda _{\mathrm {max}} ^t \Vert w_1 \Vert )^{2(d-1)}}{(2 \Vert w_1 \Vert )^{2(d-1)}} \varphi (w)
\geqslant c_1 \lambda _{\mathrm {max}} ^{2t(d-1)},
\end{align*}
where $c_1 = 2 ^{-2(d-1)} c_0$ and $c_0$ comes from the previous lemma.
On the other hand, we have
$$\frac{\Vert v_1 \Vert}{\Vert v' \Vert} = \frac{\Vert \lambda _{\mathrm {max}}^t w_1 \Vert}{\Vert g^t(w') \Vert}
\leqslant \frac{\lambda _{\mathrm {max}} ^t}{\lambda _{\mathrm {min}}^t} \frac{\Vert w_1 \Vert}{\Vert w' \Vert} =
\frac{\lambda _{\mathrm {max}} ^t}{\lambda _{\mathrm {min}}^t}.$$ 
Since $(\lambda _{\mathrm {max}})^{2(d-1)} = (\lambda _{\mathrm {max}} / \lambda _{\mathrm {min}})^\alpha$ we obtain by combining the above inequalities 
$$\varphi (v) \geqslant c_1 \lambda _{\mathrm {max}} ^{2t(d-1)} = c_1 \frac{\lambda _{\mathrm {max}}^{t\alpha}}{ \lambda _{\mathrm {min}}^{t\alpha}}
\geqslant c_1 \frac{\Vert v_1 \Vert ^\alpha}{\Vert v' \Vert ^\alpha}.$$
Now there is a neighborhood of $v_0$ in which the ratio $\Vert v_1 \Vert ^\alpha / \Vert v' \Vert ^\alpha$ is comparable to 
$1 / \Vert v' \Vert ^\alpha$, which in turn is larger that $1/ \Vert v - v_0 \Vert ^\alpha$; the lemma follows.  
\end{proof}
Finally we complete the 
\begin{proof}[Proof of Corollary \ref{intro_cor2}]
Let $u \in B_{p,p} ^{(d-1)/p} ({\bf R} ^{d-1})$ be a function such that $\int _M \!\Vert u \circ m \Vert _B ^p d\mathcal H (m) 
<  \infty$. We want to prove that $u$ is constant \emph{a.e.} According to Lemma \ref{H^1_lem0}, it is enough to show that there is a non-zero vector 
$v_0 \in {\bf R}^{d-1}$ 
such that $u$ is constant almost everywhere along almost every line directed by $v_0$. 

Recall that $G$ is assumed to have rank $\geqslant 2$; since $P$ is a maximal proper parabolic subgroup, the Levi factor $M$ has rank $\geqslant 1$, hence is a noncompact semisimple group. 
Therefore $M$ contains a 
$1$-parameter subgroup $(g^t)_{t \in {\bf R}}$ which acts on ${\bf R} ^{d-1}$ as in Lemma \ref{H^1_lem2}. Indeed $M$ contains a Lie subgroup $H$ which is virtually
isomorphic to $\mathrm{SL}(2, \bf R)$; and from the representation theory of $\mathrm{SL}(2, \bf R)$, one knows that $H$ contains such a $1$-parameter subgroup
(see \cite[Cor.\! 7.2]{Humphreys-LA}).

Let $\lambda _{\mathrm {max}}$,
$\lambda _{\mathrm {min}}$, $v_0$, $c_2$, $U$ and $\alpha$ be as in Lemma \ref{H^1_lem2}. By changing $g$ 
to its inverse, if necessary, we can assume that $\lambda _{\mathrm {max}} \geqslant \lambda _{\mathrm {min}}^{-1}$; so that we have 
$\alpha \geqslant d-1$.  Since $u$ is measurable, it is approximately continuous \emph{a.e.} \!(see \cite[2.9.13]{Fed}).
In other words, for \emph{a.a.}\! $x \in {\bf R} ^{d-1}$ and every $\varepsilon >0$,
we have 
$$\lim _{r \to 0} \mathrm{meas}\bigl\{x' \in B(x,r) ~;~ \vert u(x') - u(x) \vert < \varepsilon \bigr\} \big/ \mathrm{meas}\bigl(B(x,r)\bigr) =1.$$
Let $a, b$ be distinct points on a line directed by $v_0$, such that $u$ is approximately
continuous at $a$ and $b$. We want to show that $u(a) = u(b)$. By multiplying $v_0$ by a real number if necessary, we can assume that 
$v_0 = b-a$.
For every $\varepsilon > 0$, there exists $r_\varepsilon >0$ such that for every $0<r \leqslant r_\varepsilon$ the set 
$$E_r (a) := \bigl\{x \in B(a,r) ~;~ \vert u(x) - u(a) \vert < \varepsilon \bigr\}$$
contains at least $99 \%$ of the measure of $B(a,r)$. 
We define similarly $E_r (b)$.

Suppose by contradiction that $u(a) \neq u(b)$. Let $\varepsilon = \vert u(a) - u(b) \vert /4$ and $0<r \leqslant r_\varepsilon$.
Then for every $(x,x') \in E_r (a) \times E_r (b)$ one has $\vert u(x) - u(x') \vert \geqslant \vert u(a) - u(b) \vert /2$.
Therefore by making the change of variable $v = x' -x$, we obtain with equation \ref{H^1_eq0}:
\begin{align*}
&\int _M \!\Vert u \circ m \Vert _B ^p d\mathcal H (m) \geqslant 
\int _{E_r (a) \times E_r (b)} \frac{\vert u(x) - u(x') \vert ^p}{\Vert x-x' \Vert ^{2(d-1)}}
\varphi (x-x') dxdx'\\
&\geqslant c_3 \vert u(a) - u(b) \vert ^p \int _{E_r (a) \times E_r (b)} \varphi (x-x') dxdx'\\
&= c_3 \vert u(a) - u(b) \vert ^p \int _{v \in {\bf R} ^{d-1}} \mathrm{meas}\bigl\{E_r (a) \cap \bigl(E_r (b) -v \bigr)\bigr\} 
 \varphi (v) dv,
\end{align*}
where $c_3$ is a positive constant that depends only on $\Vert a - b \Vert$. Since $b-a = v_0$, for $v$ 
close enough to $v_0$, one has 
$$\mathrm{meas}\bigl\{E_r (a) \cap \bigl(E_r (b) -v\bigr)\bigr\} \geqslant (3/10) \mathrm{meas}\bigl(B(a, r)\bigr).$$
With Lemma \ref{H^1_lem2} and the inequality $\alpha \geqslant d-1$ one gets that for every small enough neighbourhood $U$ of $v_0$:
$$
\int _M \!\Vert u \circ m \Vert _B ^p d\mathcal H (m) \geqslant 
c_4 \vert u(a) - u(b) \vert ^p \int _{U} \frac{dv}{\Vert v-v_0 \Vert ^{d-1}},$$
where $c_4$ depends only on $c_2, c_3, r$. Since the last integral is equal to $+\infty$, we obtain a contradiction.
\end{proof}

\def\cprime{$'$}

\bigskip

\bigskip

\noindent Laboratoire Paul Painlev\'e, UMR 8524 de CNRS, Universit\'e de Lille, 
Cit\'e Scientifique, Bat. M2, 59655 Villeneuve d'Ascq,
France. \\E-mail: bourdon@math.univ-lille1.fr.

\noindent Centre de Math\'ematiques Laurent Schwartz, UMR 7640 de CNRS, \'Ecole polytechnique,
91128 Palaiseau, France. \\E-mail: bertrand.remy@polytechnique.edu.


\begin{thebibliography}{BFGM07}

\bibitem[BFGM07]{BFGM}
Y.~Bader, A.~Furman, T.~Gelander, and N.~Monod: 
\newblock Property ({T}) and rigidity for actions on {B}anach spaces.
\newblock {\em Acta Math.}, 198(1):57--105, 2007.

\bibitem[BFS14]{BFS}
Y.~Bader, A.~Furman, and R.~Sauer: 
\newblock Weak notions of normality and vanishing up to rank in
  ${L}^2$-cohomology.
\newblock {\em Int. Math. Res. Not. IMRN}, 12:3177--3189, 2014.

\bibitem[Bla79]{Bl}
Ph. Blanc: 
\newblock Sur la cohomologie continue des groupes localement compacts.
\newblock {\em Ann. Sci. Ecole Norm. Sup.}, 12(2):137--168, 1979.

\bibitem[Bor69]{Borel-Arith}
A.~Borel: 
\newblock  {\em Introduction aux groupes arithm\'etiques}.
\newblock Actualit\'es Scientifiques et Industrielles, No. 1341. Hermann, Paris, 1969.

\bibitem[Bor85]{Borel}
A.~Borel: 
\newblock The ${L}^2$-cohomology of negatively curved {R}iemennian symmetric
  spaces.
\newblock {\em Ann. Acad. Sci. Fenn. Ser. A I Math.}, 10:95--105, 1985.

\bibitem[Bor91]{Borel-LAG}
A.~Borel: 
\newblock  {\em Linear algebraic groups}. Second edition. 
\newblock Graduate Texts in Mathematics, 126. Springer-Verlag, New York, 1991.

\bibitem[BoTi65]{BoTi}
A.~Borel, J.~Tits: 
\newblock Groupes r\'eductifs.
\newblock {\em Publ. Math. Inst. Hautes \'Etudes Sci.}, 27:55--150, 1965.

\bibitem[Bou63]{BBK-INT-7}
N.~Bourbaki: 
\newblock {\em {I}nt\'egration. {C}hapitres 7
  \`a 8}.
\newblock Actualit\'es Scientifiques et Industrielles, No. 1306. Hermann,
  Paris, 1963.

\bibitem[Bou65]{BBK-INT-1to4}
\bysame: 
\newblock {\em {I}nt\'egration. {C}hapitres 1
  \`a 4}.
\newblock Actualit\'es Scientifiques et Industrielles, No. 1175. Hermann,
  Paris, 1965.

\bibitem[Bou68]{BBK-Lie-4to6}
\bysame: 
\newblock {\em {G}roupes et alg\`ebres de
  {L}ie. {C}hapitres 4 \`a 6}.
\newblock Actualit\'es Scientifiques et Industrielles, No. 1337. Hermann,
  Paris, 1968.

\bibitem[Bou07]{BBK-TG-5to10}
\bysame: 
\newblock {\em {T}opologie g\'en\'erale.
  {C}hapitres 5 \`a 10}.
\newblock Springer, 2007.

\bibitem[Bou16]{B16}
M.~Bourdon: 
\newblock Cohomologie $\ell _p$ en degr\'es sup\'erieurs et dimension conforme.
\newblock {\em Ann. Inst. Fourier}, 66:1013--1043, 2016.

\bibitem[BP03]{BP2}
M.~Bourdon and H.~Pajot: 
\newblock Cohomologie {$\ell \sb p$} et espaces de {B}esov.
\newblock {\em J. Reine Angew. Math.}, 558:85--108, 2003.

\bibitem[BS87]{BS}
K.~Burns, R.~Spatzier: 
\newblock Manifolds of nonpositive curvature and their buildings.
\newblock {\em Publ. Math. Inst. Hautes \'Etudes Sci.}, 65:35--59, 1987. 

\bibitem[BW00]{BW}
A.~Borel and N.~Wallach: 
\newblock {\em Continuous cohomology, discrete subgroups, and representations
  of reductive groups. {S}econd edition}.
\newblock American Mathematical Society, Providence, RI, 2000.
\newblock Mathematical Surveys and Monographs, 67.

\bibitem[CdlH16]{CdH}
Y.~Cornulier and P.~de~la Harpe: 
\newblock {\em Metric geometry of locally compact groups}, volume~25 of {\em
  EMS Tracts in Mathematics}.
\newblock European Mathematical Society (EMS), Z\"urich, 2016.

\bibitem[CT11]{CT}
Y.~Cornulier and R.~Tessera: 
\newblock Contracting automorphisms and ${L}^p$-cohomology in degree one.
\newblock {\em Ark. Mat.}, 49(2):295--324, 2011.

\bibitem[Ele98]{Elek}
G.~Elek: 
\newblock Coarse cohomology and $\ell^p$-cohomology. 
\newblock {\em K-Theory}, 13: 1-22, 1998.


\bibitem[Fed69]{Fed}
H.~Federer: 
\newblock {\em Geometric measure theory}, volume 153 of {\em Die Grundlehren
  der mathematischen Wissenschaften}.
\newblock Springer-Verlag, New-York, 1969.

\bibitem[Gen14]{Genton}
L.~Genton: 
\newblock {\em Scaled {A}lexander-{S}panier {C}ohomology and $L^{q,p}$
  {C}ohomology for {M}etric {S}paces}.
\newblock These no 6330. EPFL, Lausanne, 2014.

\bibitem[Gro93]{G}
M.~Gromov: 
\newblock Asymptotic invariants of infinite groups.
\newblock In {\em Geometric group theory, {V}ol.\ 2 ({S}ussex, 1991)}, volume
  182 of {\em London Math. Soc. Lecture Note Ser.}, pages 1--295. Cambridge
  Univ. Press, Cambridge, 1993.

\bibitem[Gui80]{Gui}
A.~Guichardet: 
\newblock {\em Cohomologie des groupes topologiques et des alg\`ebres de
  {L}ie}.
\newblock Textes Math\'ematiques. CEDIC/Fernand Nathan, Paris, 1980.

\bibitem[Hei01]{Hei}
J.~Heinonen: 
\newblock {\em Lectures on analysis on metric spaces},
{\em Universitext}.
\newblock Springer-Verlag, New York, 2001.

\bibitem[Hel01]{Helgason}
S.~Helgason: 
\newblock {\em Differential geometry, {L}ie groups, and symmetric spaces},
  volume~34 of {\em Graduate Studies in Mathematics}.
\newblock American Mathematical Society, Providence, RI, 2001.
\newblock Corrected reprint of the 1978 original.

\bibitem[Hum72]{Humphreys-LA}
J.~E. Humphreys: 
\newblock {\em Introduction to {L}ie algebras and representation theory}.
\newblock Springer, 1972.
\newblock Graduate Texts in Mathematics, Vol. 9.

\bibitem[Kna02]{Knapp}
Anthony W.~Knapp
\newblock {\em Lie groups beyond an introduction}. Second edition. 
\newblock Progress in Mathematics, 140. Birkhäuser Boston, Inc., Boston, MA, 2002

\bibitem[Mau09]{Maubon}
J.~Maubon
\newblock {\em Symmetric spaces of the non-compact type: differential geometry}. 
\newblock In G\'eom\'etries \`a courbure n\'egative ou nulle, groupes discrets et rigidit\'es, 1--38, 
S\'emin. Congr., 18, Soc. Math. France, Paris, 2009

\bibitem[Mic56]{Mi}
E.~Michael: 
\newblock Continuous selections {I}.
\newblock {\em Ann. of Maths.}, 63:361--382, 1956.

\bibitem[Mon01]{Monod_LNM}
N.~Monod: 
\newblock {\em Continuous bounded cohomology of locally compact groups}, volume
  1758 of {\em Lecture Notes in Mathematics}.
\newblock Springer-Verlag, Berlin, 2001.

\bibitem[Mon10]{Monod}
\bysame: 
\newblock On the bounded cohomology of semi-simple groups, ${S}$-arithmetic
  groups and products.
\newblock {\em J. Reine Angew. Math.}, 640(1):97--128, 2010.

\bibitem[Pan89]{Pa1}
P.~Pansu: 
\newblock Cohomologie {$L\sp p$} des vari\'et\'es \`a courbure n\'egative, cas
  du degr\'e {$1$}.
\newblock {\em Rend. Sem. Mat. Univ. Politec. Torino}, (Special Issue):95--120,
  1989.
\newblock Conference on Partial Differential Equations and Geometry (Torino,
  1988).

\bibitem[Pan95]{Pa95}
\bysame: 
\newblock Cohomologie ${L}^p$ : invariance sous quasiisom\'etrie.
\newblock Preprint 1995.

\bibitem[Pan99]{Pa99}
\bysame: 
\newblock Cohomologie ${L}^p$, espaces homog\`enes et pincement.
\newblock Preprint 1999.

\bibitem[Pan07]{P2}
\bysame: 
\newblock Cohomologie ${L}^p$ en degr\'e $1$ des espaces homog\`enes.
\newblock {\em Potential Anal.}, 27:151--165, 2007.

\bibitem[Pan08]{P1}
\bysame: 
\newblock Cohomologie ${L}^p$ et pincement.
\newblock {\em Comment. Math. Helv.}, 83(2):327--357, 2008.

\bibitem[Rez08]{Rez}
A.~Reznikov: 
\newblock Analytic topology of groups, actions, strings and varieties.
\newblock In {\em Geometry and dynamics of groups and spaces}, volume 265 of
  {\em Progr. Math.}, pages 3--93. Birkh\"auser, Basel, 2008.

\bibitem[SS18]{SaSc}
R.~Sauer and M.~Schr\" odl: 
\newblock Vanishing of $\ell^2$-Betti numbers of locally compact groups as an invariant of coarse equivalence. 
\newblock {\em Fund. Math.}, 243(3):301--311, 2018.


\bibitem[Tes09]{T}
R.~Tessera: 
\newblock Vanishing of the first reduced cohomology with values in a $L^p$-representation.
\newblock {\em  Ann. Inst. Fourier}, 59(2):851--876, 2009.

\bibitem[Z84]{Z}
R.~Zimmer:
\newblock {\em Ergodic Theory and Semisimple Groups}.
\newblock Birkh\"auser.
\newblock Monographs in Mathematics, Vol. 81.



\end{thebibliography}
\end{document}